\documentclass[12pt]{amsart}
\usepackage{geometry}
\usepackage[mathscr]{euscript}
\usepackage{ amssymb }
\usepackage[numbers, square, comma, sort&compress]{natbib}
\usepackage{bm}
\usepackage{enumerate}
\usepackage[english]{babel}
\usepackage{commath}
\usepackage{tikz-cd}
\usepackage{graphicx}
\usepackage{filecontents}
\usepackage{csquotes}[autostyle, english = american]

\theoremstyle{plain}
\newtheorem{theorem}{Theorem}[section]

\theoremstyle{plain}
\newtheorem{lemma}[theorem]{Lemma}

\theoremstyle{plain}
\newtheorem{corollary}[theorem]{Corollary}

\theoremstyle{definition}
\newtheorem{definition}[theorem]{Definition}

\theoremstyle{plain}
\newtheorem{proposition}[theorem]{Proposition}

\theoremstyle{remark}
\newtheorem{remark}[theorem]{Remark}

\theoremstyle{definition}
\newtheorem{example}[theorem]{Example}

\theoremstyle{plain}
\newtheorem{claim}[theorem]{Claim}

\theoremstyle{plain}

\theoremstyle{plain}
\newtheorem{question}[theorem]{Question}
\newcommand\blfootnote[1]{%
  \begingroup
  \renewcommand\thefootnote{}\footnote{#1}%
  \addtocounter{footnote}{-1}%
  \endgroup
}

\title[]{Relating Asymptotic Dimension to Ponomarev's Cofinal Dimension via Coarse Proximities}

\author{Jeremy Siegert}
\address{Ben Gurion University of the Negev, Beer Sheva, Israel} 
\email{siegertj@post.bgu.ac.il}
\date{\today} 


\begin{document}

\begin{abstract}
In this paper we show that the asymptotic dimension of an unbounded proper metric space is bounded above by a coarse analog of Ponomarev's cofinal dimension of topological spaces, which we call the coarse cofinal dimension. We also show that asymptotic dimension is bounded below by the cofinal dimension of the Higson corona by existing results of Miyata, Austin, and Virk. We do this by introducing several constructions in the theory of coarse proximity spaces. In particular we introduce the inverse limit of coarse proximity spaces.  We end with some open problems.
\end{abstract}

\maketitle
\tableofcontents

\blfootnote{The author was supported by Israel Science Foundation grant no. 2196/20}
\section{Introduction}
Coarse geometry studies metric spaces and abstract spaces from a large scale perspective where phenomena occurring in sets designated as bounded are irrelevant and one considers instead the properties that manifest as interactions between unbounded sets and are persistent through ``zooming out". One such large scale invariant is the asymptotic dimension defined originally for metric spaces by Gromov in \cite{Gromov} as a large scale analog of classical covering dimension. The use of this property would find application in the study of groups when Yu showed in \cite{Yu} that the Novikov conjecture is satisfied for groups of finite asymptotic dimension. There remain open problems directly involving asymptotic dimension that would advance progress towards other major problems. In particular, in \cite{dranishnikov} two open problems about the asymptotic dimension of geometrically finite groups are stated that, together, imply the Novikov conjecture. In \cite{burghelea} a conjecture on the asymptotic dimension of groups is stated that has implications for the Burghelea conjecture. As these problems remain open, greater investigation of asymptotic dimension is merited. A related large scale invariant is the Higson corona of a proper metric space, which is a compact Hausdorff remainder of a compactification that models the large scale behavior of a space. Dranishnikov asked in \cite{dranishnikov} if the asymptotic dimension of an unbounded proper metric space is equal to the covering dimension of its Higson corona. Dranishnikov, Keesling, and Uspenkij would show in \cite{dranishnikov} and \cite{asdimlessthandim} that if the asymptotic dimension of an unbounded proper metric space is finite, then the two dimensions do indeed agree. The case where asymptotic dimension is infinite is still open. A characterization of asymptotic dimension or sufficient criteria for finite asymptotic dimension purely in terms of topological dimensions of the Higson corona allows questions regarding asymptotic dimension to be evaluated with the much more mature techniques of classical dimension theory. 
\vspace{\baselineskip}

In this paper we make the case that asymptotic dimension is more closely related to the cofinal dimension of the Higson corona than the covering dimension. The cofinal dimension of a topological space is defined in terms of covers that are made up of closures of pairwise disjoint open sets. Such covers are called canonical covers. The cofinal dimension of spaces, and compact Hausdorff spaces in particular, can be characterized by certain finite-to-one surjections of zero dimensional spaces. A detailed explanation of cofinal dimenison can be found in \cite{pears}. Recently, Austin, Miyata, and Virk proved characterizations of asymptotic dimension using large scale analogs of surjective finite-to-one maps in \cite{austinvirk} and \cite{miyatavirk}. This makes up the motivation for the results of this paper. Using slight variations of the results of Austin, Miyata, and Virk and extant results about cofinal dimension we prove that the asymptotic dimension of unbounded proper metric spaces is bounded above by a coarse analog of Ponomarev's cofinal dimension that we call the coarse cofinal dimension. To do this we make use of coarse proximity spaces, abstract large scale objects first defined by Grzegrzolka and the author in \cite{paper1}. Coarse proximity spaces are a large scale analog of small scale proximity spaces originally introduced by Efremovic as detailed in \cite{proximityspaces}. As is characteristic of coarse properties defined using the language of coarse proximity spaces, our definition of the coarse cofinal dimension is purely in terms of unbounded sets and closeness relations between them, as opposed to the usual characterization of asymptotic dimension using uniform covers. The convenience of coarse proximity spaces comes from the boundary functor which assigns to each coarse proximity space a compact Hausdorff space that captures the geometry evident in closeness relations of the unbounded sets of the base space. In the case of proper metric spaces, the boundary of the obvious coarse proximity structure is homeomorphic to the Higson corona, making characterizing coarse properties in terms of unbounded sets very useful in recharacterizing coarse properties such as asymptotic dimension in terms of the Higson corona. 
\vspace{\baselineskip}

The progression of this paper is as follows. In Section \ref{basic constructions} we cover the basic machinery used in later sections. Much of this material is published elsewhere, but is reproduced here in the interest of self containment. Included herein are the definitions of small scale proximity spaces, the construction of the Smirnov compactification, coarse proximity spaces, the boundary functor, and cofinal dimension. While most of the definitions and results in this section appear elsewhere, there are some new results that are critical for the work in later sections. In section \ref{asymptotic dimension of coarse proximity spaces} we define and very briefly explore the asymptotic dimension of coarse proximity spaces. This is done via uniformly bounded covers of coarse proximity spaces the definition of which is an aesthetic variation of the same definition provided by Honari and Kalantari in \cite{Honari}. In the case where a coarse proximity structure comes from a metric, our definition of asymptotic dimension agrees with the usual definition. Section \ref{metrizability} is a short section in which we prove a necessary and sufficient condition for the structure of a coarse proximity space to be induced by a metric. Section \ref{inverse limits} introduces our first major construction, the inverse limit of coarse proximity spaces. Here we construct inverse limits in the categories of coarse proximity spaces with coarse proximity maps. Note that the morphisms in this category are not closeness classes, but individual functions. We show that the inverse limit construction commutes with the boundary functor. In section \ref{coarse n to 1} we review the notions and results surrounding coarsely $n$-to-$1$ maps as defined by Austin, Miyata, and Virk. We also prove some generalizations to coarse proximity spaces of their results. It is also in this section that we introduce the coarse cofinal dimension, the large scale analog of Ponomarev's cofinal dimension. Finally in section \ref{equality of coarse cofinal and asdim} we provide a characterization of the coarse cofinal dimension that demonstrates that the dimension is at least the asymptotic dimension in the case of proper metric spaces. In the final section we conclude with some open problems.

\section{Basic constructions}\label{basic constructions}

In this first preliminary section we will review many of the basic structures and definitions used throughout the rest of the paper. Beginning with the basic theory of coarse proximity spaces we review the definitions of bornologies, coarse proximity spaces, maps between coarse proximity spaces, boundaries of coarse proximity spaces, and the basic properties of those boundaries. This material is presented roughly as it was across \cite{paper1}, \cite{paper2}, \cite{paper3}, and \cite{thesis}. However there are some new results. After finishing the basic coarse constructions and concepts we proceed to a review of, and a new charaterization,  the cofinal dimension of topological spaces as described in \cite{pears}. The definitions in this subsection are of critical importance to later sections.

\subsection{Small-scale Proximity Spaces and the Smirnov compactification}

The presentation of small-scale proximity spaces here is as seen in \cite{proximityspaces}.

\begin{definition}
A binary relation $\delta$ on the power set of a nonempty set $X$ is called a {\bf proximity relation} and the pair $(X,\delta)$ a {\bf proximity space} if the relation satisfies:\\
\begin{enumerate}
\item For all $A\subseteq X$, $A\bar{\bf \delta}\emptyset$.
\item $A\delta B$ implies $B\delta A$.
\item $A\cap B\neq\emptyset$ implies $A\delta B$.
\item $A\delta(B\cup C)$ if and only if $A\delta B$ or $A\delta C$.
\item $A\bar{\delta}B$ implies that there is an $E\subseteq X$ such that $A\bar{\delta}(X\setminus E)$ and $E\bar{\delta}B$.
\end{enumerate}

\noindent
Where $A\bar{\delta}B$ denotes the statement that ``$A\delta B$ does not hold". If in addition the relation satisfies:

\[\{x\}\delta\{y\}\iff x=y\]

\noindent
then we say that the proximity $\delta$ and the corresponding proximity space $(X,\delta)$ are {\bf separated}.
\end{definition}

\begin{proposition}
The {\bf induced topology} of a proximity space $(X,\delta)$ is the topology on $X$ determined by the closure operator defined by

\[cl(A)=\{x\in X\mid \{x\}\delta A\}\]

\noindent
This topology is Hausdorff if and only if the proximity is separated.
\end{proposition}

\begin{example}
If $(X,d)$ is a metric space, then the relation $\delta$ defined by $A\delta B$ if and only if $d(A,B)=0$ is a proximity relation on $X$, whose induced topology is the topology induced by the metric $d$. We call this proximity the {\bf metric proximity}.
\end{example}

\begin{example}
If $X$ is a compact Hausdorff space, then the relation $\delta$ defined by $A\delta B$ if and only if their closures intersect is a proximity relation on $X$ that induces the original topology on $X$. This proximity relation is in fact the unique proximity relation on $X$ that induces the original topology.
\end{example}

\begin{example}
If $(X,d)$ is a proper metric space (that is, a metric space in which all closed and bounded sets are compact), then the relation $\delta$ defined on $X$ by

\[A\delta B\iff\left\{\begin{array}{ll}
d(A,B)=0 & \text{ or }
\lim_{r\to\infty}d(A\setminus B(x_{0},r),B\setminus B(x_{0},r))<\infty
\end{array}
\right.\]

\noindent
where $x_{0}$ is any point in $X$, is a proximity relation on $X$ that induces the metric topology induced by $d$. We call this proximity the {\bf Higson proximity}.

\end{example}

\begin{example}
If $(X,\delta)$ is a proximity space and $Y\subseteq X$ is a nonempty set, then the relation $\delta_{Y}$ defined on subsets of $Y$ by

\[A\delta_{Y}B\iff A\delta B\]

\noindent
is a proximity relation on $Y$ that induces the subspace topology on $Y$ inherited from the topology on $X$ induced by $\delta$. We of course call this proximity the {\bf subspace proximity}.

\end{example}

\begin{definition}
A function $f:(X,\delta_{X})\rightarrow (Y,\delta_{Y})$ is called a {\bf proximity map} if for all $A,B\subseteq X$ we have

\[A\delta_{X}B\implies f(A)\delta_{Y} f(B)\]

\noindent
We call $f$ a {\bf proximity isomorphism} if $f$ is bijective and the inverse function $g:Y\rightarrow X$ is also a proximity map. 
\end{definition}

Proximity maps are continuous with respect to the topologies induced by the proximity relations. The collection of proximity spaces together with proximity maps makes up the category {\bf Prox}. 

\begin{definition}
A nonempty collection $\sigma$ of subsets of a proximity space $(X,\delta)$ is called a {\bf cluster} if it satisfies the following:\\
\begin{enumerate}
\item For all $A,B\in\sigma$, $A\delta B$.
\item If $(A\cup B)\in\sigma$ then $A\in\sigma$ or $B\in\sigma$.
\item If $C\delta A$ for all $A\in\sigma$, then $C\in\sigma$.
\end{enumerate}

\noindent
If $\sigma$ contains $\{x\}$ for some $x\in X$, we call $\sigma$ a {\bf point cluster} and denote it by $\sigma_{x}$.
\end{definition}

Note that in a separated proximity space, $\sigma_{x}=\sigma_{y}$ if and only if $x=y$.
\vspace{\baselineskip}

\begin{definition}
Given a collection of clusters $\mathcal{A}$ of a proximity space $(X,\delta)$ we say that a subset $A\subseteq X$ {\bf absorbs} $\mathcal{A}$ if $A\in\sigma$ for all $\sigma\in\mathcal{A}$.
\end{definition}

Given a separated proximity space $(X,\delta)$ the Smirnov compactification of $(X,\delta)$ is constructed in the following way. Let $\mathfrak{X}$ be the set of all clusters in $X$. We define a relation $\delta^{*}$ on the collection of subsets of $\mathfrak{X}$ as follows. If $\mathcal{A},\mathcal{B}\subseteq\mathfrak{X}$ then, $\mathcal{A}\delta^{*}\mathcal{B}$ if and only if for all absorbing sets $A$ for $\mathcal{A}$, and $B$ for $\mathcal{B}$, we have that $A\delta B$. Then $(\mathfrak{X},\delta^{*})$ is a compact separated proximity space. Moreover, the map $f:X\rightarrow\mathfrak{X}$ defined by $f(x)=\sigma_{x}$ is a proximity (and consequently topological) embedding whose image is dense in $\mathfrak{X}$. The proximity $\delta$ on $X$ can be recharacterized by saying that $A\delta B$ if and only if $cl_{\mathfrak{X}}(A)\cap cl_{\mathfrak{X}}(B)\neq\emptyset$. The Smirnov compactification is the unique compact Hausdorff space into which $(X,\delta)$ proximity embeds as a dense subspace. 
\vspace{\baselineskip}

\begin{proposition}\label{associated cluster}
	Let $(X, \delta_1)$ and $(Y, \delta_2)$ be proximity spaces. Let $f:X\rightarrow Y$ be a proximity map. Then to each cluster $\sigma_1$ in $X,$ there corresponds a cluster $\sigma_2$ in $Y$ given by
	\[\sigma_2:=\{A\subseteq Y\mid A\delta_2 f(B)\text{ for all }B\in\sigma_1\}.\]
\end{proposition}

\begin{theorem}\label{extension_theorem}
	Let $(X, \delta_1)$ and  $(Y, \delta_2)$ be separated proximity spaces and $(\mathfrak{X}, \delta_1^{*})$ and $(\mathfrak{Y}, \delta_2^{*})$ the respective Smirnov compactifications. Let $f: X \to Y$ be a proximity map. Then $f$ extends to a unique proximity map $\bar{f}:\mathfrak{X} \to \mathfrak{Y}$ where $\sigma_{1}\in\mathfrak{X}$ is mapped to the associated cluster $\sigma_{2}\in\mathfrak{Y}$ as described in Proposition \ref{associated cluster}.
\end{theorem}

The assignment of the Smirnov compactification to a separated proximity space and the assignment of the extended map between Smirnov compactifications to a proximity map makes up a functor from {\bf Prox} to the category of compact Hausdorff spaces, {\bf CompHs}.

\subsection{Coarse Proximity Spaces} 
The methods we will use in later sections to prove our main results are centered around the use of coarse proximity spaces, a structure in coarse geometry first introduced in \cite{paper1}. Here we will review the important basic definitions and results. They are as they appear in \cite{paper1} and \cite{paper3}.
\vspace{\baselineskip}

\begin{definition}
	A {\bf bornology} $\mathcal{B}$ on a set $X$ is a family of subsets of $X$ satisfying:
	\begin{enumerate}
		\item $\{x\}\in\mathcal{B}$ for all $x\in X,$
		\item $A\in\mathcal{B}$ and $B\subseteq A$ implies $B\in\mathcal{B},$
		\item If $A,B\in\mathcal{B},$ then $A\cup B\in\mathcal{B}.$
	\end{enumerate}
	Elements of $\mathcal{B}$ are called {\bf bounded} and subsets of $X$ not in $\mathcal{B}$ are called {\bf unbounded}. 
\end{definition}

\begin{definition}\label{coarseproximitydefinition}
	Let $X$ be a set equipped with a bornology $\mathcal{B}$. Let $A,B,$ and $C$ be subsets of $X$. A \textbf{coarse proximity} on a set $X$ is a relation ${\bf b}$ on the power set of $X$ satisfying the following axioms:
	
	\begin{enumerate}
		\item $A{\bf b}B$ implies $B{\bf b}A,$ \label{axiom1}
		\item $A{\bf b}B$ implies $A \notin \mathcal{B}$ and $B \notin \mathcal{B},$ \label{axiom2}
		\item $A\cap B \notin \mathcal{B}$ implies $A {\bf b} B,$ \label{axiom3}
		\item $(A \cup B){\bf b}C$ if and only if $A{\bf b}C$ or $B{\bf b}C,$ \label{axiom4}
		\item $A\bar{\bf b}B$ implies that there exists a subset $E$ such that $A\bar{\bf b}E$ and $(X\setminus E)\bar{\bf b}B,$ \label{axiom5}
	\end{enumerate}
	where $A\bar{ {\bf b}}B$ means ``$A{\bf b} B$ is not true." If $A {\bf b} B$, then we say that $A$ is \textbf{coarsely close} to (or \textbf{coarsely near}) $B.$ Axiom (\ref{axiom4}) will be called the \textbf{union axiom} and axiom (\ref{axiom5}) will be called the \textbf{strong axiom}. A triple $(X,\mathcal{B},{\bf b})$ where $X$ is a set, $\mathcal{B}$ is a bornology on $X$, and ${\bf b}$ is a coarse proximity relation on $X,$ is called a {\bf coarse proximity space}.
\end{definition}

\begin{example}
	Let $(X,d)$ be a metric space with the bornology consisting of all the metrically bounded sets. Define two subsets $A$ and $B$ of $X$ to be coarsely close if and only if there exists $\epsilon < \infty$ such that for all bounded sets $D$, there exists $a \in (A \setminus D)$ and $b \in (B \setminus D)$ such that $d(a,b) < \epsilon.$  Then this relation is a coarse proximity, called \textbf{the metric coarse proximity.}
\end{example}

\begin{example}
Let $X$ be a Hausdorff topological space and $Z$ a compact Hausdorff space containing $X$. Define $\mathcal{B}$ to be the collection of subsets of $X$ that are closed as subsets of $Z$. For subsets $A,C\subseteq X$ not in $\mathcal{B}$, define $A{\bf b}C$ if and only if $cl_{Z}(A)\cap cl_{Z}(C)\cap(Z\setminus X)\neq\emptyset$. Then $(X,\mathcal{B},{\bf b})$ is a coarse proximity space. 
\end{example}

\begin{definition}
Given subsets $A,C$ of a coarse proximity space $(X,\mathcal{B},{\bf b})$ we say that $A$ is a {\bf coarse neighbourhood} of $C$, denoted $C\ll A$, if $C\bar{\bf b}(X\setminus A)$. 
\end{definition}

It was shown in \cite{paper2} that coarse proximity spaces can be characterized entirely by the relation $\ll$. 

\begin{definition}\label{weakasymptoticresemblance}
	Let $X$ be a set and $\phi$ an equivalence relation on the power set of $X$ satisfying the following property:	
	\[A\phi B \text{ and } C\phi D \quad \quad \Longrightarrow \quad \quad (A\cup C)\phi(B\cup D).\]
	Then we call $\phi$ a \textbf{weak asymptotic resemblance}. If $A\phi B,$ then we say that $A$ and $B$ are \textbf{$\phi$ related}.
\end{definition}

Weak asymptotic resemblances are relaxations of the stronger asymptotic resemblances defined by Honari and Kalantari in \cite{Honari}. Any coarse proximity naturally induces a weak asymptotic resemblance, as in the following theorem:

\begin{theorem}\label{asymptoticresemblance}
	Let $(X,\mathcal{B},{\bf b})$ be a coarse proximity space. Let $\phi$ be the relation on the power set of $X$ defined in the following way: $A\phi B$ if and only if the following hold:
	\begin{enumerate}
		\item for every unbounded $B^{\prime}\subseteq B$ we have $A{\bf b}B^{\prime},$
		\item for every unbounded $A^{\prime}\subseteq A$ we have $A^{\prime}{\bf b}B.$
	\end{enumerate}	
	Then  $\phi$ is a weak asymptotic resemblance that we call the \textbf{weak asymptotic resemblance induced by the coarse proximity ${\bf b}$.}
\end{theorem}
\begin{proof}
	See \cite{paper1}.
\end{proof}

Notice that if $\phi$ is the relation defined in Theorem \ref{asymptoticresemblance} and $A$ and $B$ are bounded, then they are always $\phi$ related. If $A$ is bounded and $B$ unbounded, then they are not $\phi$ related. If $A, B$ and $C$ are any subsets of $X$ such that $A \phi B,$ then $A{\bf b}C$ if and only if $B{\bf b}C.$ If $\phi$ is induced by the metric coarse proximity, then for nonempty subsets $\phi$  is the relation of having finite Hausdorff distance. For the proofs of the above statements, see \cite{paper1}.

\begin{remark}
As the weak asymptotic resemblance $\phi$ induced on a coarse proximity space $(X,\mathcal{B},{\bf b})$ is completely determined by ${\bf b}$ we will often use it without explicitly mentioning what it is. When it is not clear which space or which coarse proximity structure it is induced by we will index it by the space or coarse proximity structure. That is, if the weak asymptotic resemblance $\phi$ is induced on the coarse proximity space $(X,\mathcal{B},{\bf b})$ we may denote this relation as $\phi_{X}$ or $\phi_{(X,\mathcal{B},{\bf b})}$ as needed to avoid confusion. Of course, when the coarse proximity space is understood, we will not index the relation at all. 
\end{remark}

\begin{definition}
	Let $(X,\mathcal{B}_{1},{\bf b}_{1})$ and $(Y,\mathcal{B}_{2},{\bf b}_{2})$ be coarse proximity spaces. Let $f:X\rightarrow Y$ be a function. Then $f$ is a \textbf{coarse proximity map} provided that the following are satisfied for all $A,B \subseteq X$:
	\begin{enumerate}
		\item $B\in\mathcal{B}_{1}$ implies $f(B)\in\mathcal{B}_{2},$
		\item $A{\bf b}_{1}B$ implies $f(A){\bf b}_{2}f(B).$
	\end{enumerate}
\end{definition}

\begin{definition}\label{coarsecloseness}
	Let $X$ be a set and $(Y,\mathcal{B},{\bf b})$ a coarse proximity space. Two functions $f,g:X \to Y$ are {\bf close}, denoted $f\sim g$, if for all $A \subseteq X$
	\[f(A)\phi g(A),\]
	where $\phi$ is the weak asymptotic resemblance relation induced by the coarse proximity structure ${\bf b}.$
\end{definition}

\begin{definition}
	Let $(X,\mathcal{B}_{1},{\bf b}_{1})$ and $(Y,\mathcal{B}_{2},{\bf b}_{2})$ be coarse proximity spaces. We call a coarse proximity map $f: X \to Y$ a \textbf{coarse proximity isomorphism} if there exists a coarse proximity map $g:Y\to X$ such that $g\circ f\sim id_{X}$ and $f\circ g\sim id_{Y}.$ We say that $(X,\mathcal{B}_{1},{\bf b}_{1})$ and $(Y,\mathcal{B}_{2},{\bf b}_{2})$ are \textbf{isomorphic} if there exists a coarse proximity isomorphism $f:X \to Y.$
\end{definition}

The collection of coarse proximity spaces and closeness classes of coarse proximity maps makes up the category {\bf CrsProx} of coarse proximity spaces. For details, see \cite{paper1}. In \cite{paper3} a functor $\mathcal{U}$ from the category ${\bf CrsProx}$ to the category of compact Hausdorff spaces and continuous maps called the boundary functor was constructed. For a given coarse proximity space $(X,\mathcal{B},{\bf b})$ the compact Hausdorff space $\mathcal{U}X$ is called the {\bf boundary} of the coarse proximity space. It is constructed as follows.

\begin{proposition}\label{discrete extension}
Given a coarse proximity space $(X,\mathcal{B},{\bf b})$, the relation $\delta_{dis}$ on the collection of subsets of $X$ defined by

\[A\delta_{dis}B\iff\left\{\begin{array}{ll}
A\cap B\neq\emptyset & or\\
A{\bf b}B &

\end{array}
\right.\]

\noindent
is a separated proximity on $X$.
\end{proposition}

The proximity relation described above is called the {\bf discrete extension} of the coarse proximity relation ${\bf b}$. We can then take the corresponding Smirnov compactification $(\mathfrak{X},\delta_{dis}^{*})$. 

\begin{definition}
Given a coarse proximity space $(X,\mathcal{B},{\bf b})$ with corresponding discrete extension proximity space $(X,\delta_{dis})$ and Smirnov compactification $(\mathfrak{X},\delta_{dis}^{*})$, the {\bf boundary} of the coarse proximity space $X$, is the subset $\mathcal{U}X\subseteq\mathfrak{X}$ whose elements are precisely those clusters $\sigma\in\mathfrak{X}$ that do no contain any element of $\mathcal{B}$. 
\end{definition}

\begin{proposition}
For every coarse proximity space $(X,\mathcal{B},{\bf b})$ the boundary $\mathcal{U}X$ is a compact Hausdorff space.
\end{proposition}

\begin{proposition}
If $f:(X,\mathcal{B}_{X},{\bf b}_{X})\rightarrow(Y,\mathcal{B}_{Y},{\bf b}_{Y})$ is a coarse proximity maps, and $\overline{f}:\mathfrak{X}\rightarrow\mathfrak{Y}$ is the extension of $f$ to the Smirnov compactifications of the discrete extensions of $X$ and $Y$, then $\mathcal{U}f=\overline{f}\vert_{\mathcal{U}X}:\mathcal{U}X\rightarrow\mathcal{U}Y$ is well defined and continuous. Moreover, if $g,f:X\rightarrow Y$ are coarse proximity maps that are close, then $\mathcal{U}f=\mathcal{U}g$. 
\end{proposition}

The assignment of the boundary to a coarse proximity space and the assignment of $\mathcal{U}f$ to a closeness class of coarse proximity maps $[f]$ makes up a functor from the category of coarse proximity space {\bf CrsProx} to {\bf CompHs}.

\begin{example}
If $(X,d)$ is an unbounded proper metric space, then the boundary of the metric coarse proximity structure on $X$ is homeomorphic to the Higson corona of $X$. 

\end{example}

\begin{definition}
Let $(X,\mathcal{B},{\bf b})$ be a coarse proximity space, $\delta_{dis}$ the discrete extension of {\bf b}, and $\mathfrak{X}$ the corresponding Smirnov compactification. Given $A\subseteq X$ we define the {\bf trace} of $A$ on $\mathcal{U}X$ to be $cl_{\mathfrak{X}}(A)\cap\mathcal{U}X$.
\end{definition}

 In \cite{paper3} it was shown that traces are closed subsets of the boundary that characterize coarse proximities in the following way.

\begin{proposition}\label{coarse proximity characterization}
Let $A$ and $C$ be subsets of a coarse proximity space $(X,\mathcal{B},{\bf b})$, then

\begin{enumerate}
\item $tr(A)\neq\emptyset$ if and only if $A$ is unbounded
\item $A{\bf b}C\iff tr(A)\cap tr(C)\neq\emptyset$
\item $A\phi C\iff tr(A)=tr(C)$
\item $A\ll C\implies tr(A)\subseteq int(tr(C))$
\end{enumerate}

\end{proposition}

Another convenient property of the boundary that we use often and follows quickly from the Urysohn lemma is the following.

\begin{proposition}
If $(X,\mathcal{B},{\bf b})$ is a coarse proximity space with boundary $\mathcal{U}X$ and $K_{1},K_{2}\subseteq \mathcal{U}X$ are disjoint closed sets, then there are unbounded sets $A,C\subseteq X$ such that $A\bar{\bf b}C$, $K_{1}\subseteq int_{\mathcal{U}X}(tr(A))$, and $K_{2}\subseteq int_{\mathcal{U}X}(tr(C))$. 
\end{proposition}
\begin{proof}
Proof can be found in \cite{paper3}.
\end{proof}

\begin{lemma}\label{traces to coarse neighbourhoods}
Let $(X,\mathcal{B},{\bf b})$ be a coarse proximity space, $C,D\subseteq X$ unbounded sets such that $C\ll D$, and $\sigma\in int(tr(C))$. Then there is an unbounded set $K\subseteq X$ such that $\sigma\in int(tr(K))$ and $K\ll D$. 
\end{lemma}
\begin{proof}
As $\sigma\in int(tr(C))$ we can find an open subset $U\subseteq\mathcal{U}X$ that contains $\sigma$ and whose closure is completely contained in $int(tr(C))$. Using Urysohn's lemma on the Smirnov compactification $\mathfrak{X}$ of the discrete extension of $(X,\mathcal{B},{\bf b})$ there is a continuous function $f:\mathfrak{X}\rightarrow[0,1]$ that maps $\{\sigma\}$ to $0$ and $\mathcal{U}X\setminus U$ to $1$. Defining $K=f^{-1}([0,1/3))\cap X$ we have that $\sigma\in int(tr(K))\subseteq tr(K)\subseteq U\subseteq tr(C)$. It is then clear from Proposition \ref{coarse proximity characterization} that $K\ll D$.

\end{proof}

\begin{proposition}\label{preliminary proposition}
Let $(X,\mathcal{B},{\bf b})$ be a coarse proximity space with nonempty boundary $\mathcal{U}X$. If $U,V\subseteq\mathcal{U}X$ are nonempty empty sets such that $\overline{V}\subseteq U$, then there is an unbounded set $A\subseteq X$ such that $\overline{V}\subseteq int(tr(A))\subseteq tr(A)\subseteq U$ and if $C\subseteq X$ is an unbounded set such that $tr(C)\subseteq \overline{V}$, then $C\ll A$.

\end{proposition}
\begin{proof}
Let $\delta$ be the discrete extension of the coarse proximity on $X$ and let $\mathfrak{X}$ be the corresponding Smirnov compactification. As in Lemma \ref{traces to coarse neighbourhoods} we use Urysohn's lemma to get a continuous function $f:\mathfrak{X}\rightarrow[0,1]$ such that $f(\overline{V})=0$ and $f(\mathcal{U}X\setminus U)=1$. Define $A=f^{-1}([0,1/3])\cap X$. Then $\overline{V}\subseteq int(tr(A))$ and $tr(A)\subseteq U$. Now let $C\subseteq X$ be such that $tr(C)\subseteq\overline{V}$. If $C{\bf b}(X\setminus A)$, then by  Lemma \ref{coarse proximity characterization} we would have to have that $tr(C)\cap tr(X\setminus A)\neq\emptyset$. However, as $tr(C)\subseteq \overline{V}\subseteq int(tr(A))$, this is impossible. Therefore $C\ll A$.
\end{proof}

\begin{definition}
Given a coarse proximity space $(X,\mathcal{B},{\bf b})$ and a nonempty subset $Y\subseteq X$, the {\bf subspace coarse proximity structure} on $Y$ is $(Y,\hat{\mathcal{B}},\hat{\bf b})$ where

\[\hat{\mathcal{B}}=\{B\cap Y\mid B\in\mathcal{B}\}\]

\noindent
and for subsets $A,C\subseteq Y$ we define

\[A{\bf b}_{Y}C\iff A{\bf b}_{X}C\]

\end{definition}

\begin{proposition}\label{trace of subspace is its boundary}
Let $(X,\mathcal{B}_{X},{\bf b}_{X})$ be a coarse proximity space and $(Y,\mathcal{B}_{Y},{\bf b}_{Y})$ a subspace of $X$. Then $\mathcal{U}Y$ is homeomorphic to $tr_{\mathcal{U}X}(Y)$ and the inclusion map $\iota:Y\rightarrow X$ induces an embedding of $\mathcal{U}Y$ into $\mathcal{U}X$ such that $\mathcal{U}\iota(\mathcal{U}Y)=tr_{\mathcal{U}X}(Y)$.

\end{proposition}

\begin{proof}
It will suffice to show that if $\sigma\in tr_{\mathcal{U}X}(Y)$ then $\sigma$ contains an element of $\mathcal{U}Y$. Let $\sigma\in tr_{\mathcal{U}X}(Y)$ be given. Then $Y\in\sigma$. Let 

\[\hat{\sigma}=\{A\subseteq Y\mid A\in\sigma\}\]

We claim that $\hat{\sigma}$ is a cluster in $Y$. If $A,C\in\hat{\sigma}$, then $A,C\in\sigma$ which implies $A{\bf b}_{X}C$. The coarse proximity ${\bf b}_{Y}$ is simply a restriction of ${\bf b}_{X}$ to subsets of $Y$, so $A{\bf b}_{Y}C$. If $(A\cup C)\in\hat{\sigma}$, then $(A\cup C)\in\sigma$, so $A\in\sigma$ or $C\in\sigma$ which implies $A\in\hat{\sigma}$ or $C\in\hat{\sigma}$. Finally, let $A\subseteq Y$ be such such that $A{\bf b}_{Y}C$ for all $C\in\hat{\sigma}$. If $A\notin\sigma$ there is some $D\subseteq X$ such that $D\in\sigma$ and $D\bar{\bf b}_{X}A$. However, $D\in\sigma$ implies $\sigma\in tr_{\mathcal{U}X}(D)$, so we would be able to find an unbounded subset $E_{1}\subseteq X$ such that $\sigma\in int(tr_{\mathcal{U}X}(E_{1}))$ and $tr_{\mathcal{U}X}(E_{1})\cap tr_{\mathcal{U}X}(A)=\emptyset$. Using the strong axiom we can find two other such sets $E_{2},E_{3}\subseteq X$ with the additional property that $E_{1}\ll E_{2}\ll E_{3}$. We then claim that $E_{2}\cap Y$ must be unbounded. If not, then there is a bounded set $B\in\mathcal{B}_{X}$ such that $(Y\setminus B)\subseteq (X\setminus E_{2})$. Then $(Y\setminus B)\phi_{X} Y$ and we must have that $E_{1}\bar{\bf b}_{X}Y$. However, this is impossible as $Y\in\sigma$. Therefore $E_{2}\cap Y$ is unbounded. Note that $tr_{\mathcal{U}X}(Y)=tr_{\mathcal{U}X}(Y\cap E_{2})\cup tr_{\mathcal{U}X}(Y\setminus E_{2})$. Because $\sigma\notin tr_{\mathcal{U}X}(X\setminus E_{2})$ we must have that because $\sigma\in tr_{\mathcal{U}X}(Y)$ that $\sigma\in tr_{\mathcal{U}X}(Y\cap E_{2})$. Therefore $(Y\cap E_{2})\in\sigma$ and $(Y\cap E_{2})\in\hat{\sigma}$. Then, by assumption $A{\bf b}_{Y}(Y\cap E_{2})$. However, by the construction of $E_{2}$ we have that $A\bar{\bf b}_{X}(Y\cap E_{2})$, so this is impossible. We then must have that $A$ is an element of $\sigma$. We have then established that $\hat{\sigma}$ is a cluster in $Y$. It is then easy to see that $\mathcal{U}\iota(\hat{\sigma})=\sigma$ meaning $\mathcal{U}\iota(\mathcal{U}Y)=tr_{\mathcal{U}X}(Y)$. 
\vspace{\baselineskip}

The map $\mathcal{U}\iota$ is quickly seen to be injective as if $\sigma_{1},\sigma_{2}\in\mathcal{U}Y$ are distinct then there are $A\in\sigma_{1}$ and $C\in\sigma_{2}$ such that $A\bar{\bf b}_{Y}C$. Then $A\bar{\bf b}_{X}C$, so $\mathcal{U}\iota(\sigma_{1})$ can not be equal to $\mathcal{U}\iota(\sigma_{2})$, making $\mathcal{U}\iota$ injective. Then $\mathcal{U}\iota$ is a continuous injection of the compact Hausdorff $\mathcal{U}Y$ into $\mathcal{U}X$ whose image is $tr_{X}(Y)$, which makes $\mathcal{U}\iota$ and embedding whose image is exactly $tr_{X}(Y)$.  
\end{proof}

Next we will construct an analog of the disjoint union for coarse proximity spaces. Let $\{(X_{\alpha},\mathcal{B}_{\alpha},{\bf b}_{\alpha})\}_{\alpha\in S}$ be a family of coarse proximity spaces indexed by some set $S$. Define $X$ to be the disjoint union of all of the $X_{\alpha}$'s. That is, $X=\bigsqcup_{\alpha\in S}$. We then consider the following collection of sets.

\[\mathcal{B}=\{B\subseteq X\mid B\cap X_{\alpha}\in\mathcal{B}_{\alpha},\,\forall\alpha\in S\}\]

\begin{proposition}
The collection $\mathcal{B}$ is a bornology on $X$. 
\end{proposition}
\begin{proof}
If $x\in X$, then $x\in X_{\alpha}$ for some $\alpha\in S$. Then as $\{x\}\in\mathcal{B}_{\alpha}$ we have that $\{x\}\cap X_{\alpha}\in\mathcal{B}_{\alpha}$ and $\{x\}\cap X_{\beta}=\emptyset\in\mathcal{B}_{\beta}$ for all $\beta\neq\alpha$, so $\{x\}\in\mathcal{B}$, meaning $\mathcal{B}$ covers $X$. If $A\in\mathcal{B}$ and $B\subseteq A$, then $(B\cap X_{\alpha})\subseteq (A\cap X_{\alpha})$ for all $\alpha\in S$. As $(A\cap X_{\alpha})\in\mathcal{B}_{\alpha}$ for all $\alpha$ we have that $(B\cap X_{\alpha})\in\mathcal{B}_{\alpha}$ for each $\alpha$, meaning $B\in\mathcal{B}$. Finally if $A,B\in\mathcal{B}$, then $(A\cup B)\cap X_{\alpha}=(A\cap X_{\alpha})\cup(B\cap X_{\alpha})$, which yields $(A\cup B)\cap X_{\alpha}\in\mathcal{B}_{\alpha}$ for each $\alpha$ and $(A\cup B)\in\mathcal{B}$. Therefore $\mathcal{B}$ is a bornology on $X$.
\end{proof}

\begin{proposition}
The relation ${\bf b}$ defined on subsets of $(X,\mathcal{B})$ defined by

\[A{\bf b}B\iff \exists\alpha\in S,\,(A\cap X_{\alpha}){\bf b}_{\alpha}(B\cap X_{\alpha})\]

\noindent
is a coarse proximity on $X$.
\end{proposition}
\begin{proof}
The only axioms that aren't clear are the union axiom and the strong axiom. Let $A,B,C\subseteq X$ be such that $(A\cup B){\bf b}C$. Then there is an $\alpha\in S$ such that $[(A\cup B)\cap X_{\alpha}]{\bf b}_{\alpha}(C\cap X_{\alpha})$. Then $(A\cap X_{\alpha}){\bf b}_{\alpha}(C\cap X_{\alpha})$ or $(B\cap X_{\alpha}){\bf b}_{\alpha} C$. In either case we have that one of $A$ or $B$ is related to $C$ by ${\bf b}$. The converse is trivial. 
\vspace{\baselineskip}

Now let $A,C\subseteq X$ be such that $A\bar{\bf b}C$. Then for each $\alpha\in S$ we have that $(A\cap X_{\alpha})\bar{\bf b}_{\alpha}(C\cap X_{\alpha})$. For each $\alpha\in S$ let $E_{\alpha}\subseteq X_{\alpha}$ be such that $(A\cap X_{\alpha})\bar{\bf b}_{\alpha}(X_{\alpha}\setminus E_{\alpha})$ and $E_{\alpha}\bar{\bf b}_{\alpha}(C\cap X_{\alpha})$. Define $E=\bigcup_{\alpha\in S}E_{\alpha}$. Then $A\bar{\bf b}(X\setminus E)$ and $E\bar{\bf b}C$. Therefore ${\bf b}$ is a coarse proximity on $(X,\mathcal{B})$. 
\end{proof}

\begin{proposition}\label{boundary of disjoint union is disjoint union}
If the index set $S$ is finite in the above construction with $S=\{1,\ldots,n\}$, then $\mathcal{U}X$ is homeomorphic to the disjoint union of the $\mathcal{U}X_{i}$. 
\end{proposition}
\begin{proof}
We will consider each $X_{i}$ as a subset of $X$. For each $A\subseteq X$ let $A_{i}=A\cap X_{i}$. For a given $\sigma\in\mathcal{U}X$ we claim that there is a unique $i\in S$ such that

\[\sigma_{i}=\{A_{i}\mid A\in\sigma\}\]

\noindent
is an element of $\mathcal{U}X_{i}$. First let $A\in\sigma$ be given. Say, $A=\left(\bigcup A_{i}\right)$. Then, by the definition of a cluster we must have that there is a $j$ for which $A_{j}\in\sigma$. In fact, there is only one such index as if $i\neq j$ is such that $A_{j}\in\sigma$ then we would have that $A_{i}{\bf b}A_{j}$, which is a contradiction. Therefore, for each $A\in\sigma$ there is a unique $j$ for which $A_{j}\in\sigma$. By similar reasoning as above we have that there is a unique index $i$ such that $A_{i}\in\sigma$ for all $A\in\sigma$. We then have that $\sigma_{i}\subseteq\sigma$. This itself implies that $\sigma_{i}$ is a cluster in $\mathcal{U}X_{i}$. That there is no $j\neq i$ for which $\sigma_{j}\in\mathcal{U}X_{j}$ is obvious.
\vspace{\baselineskip}

Finally, define a function $f:\mathcal{U}X\rightarrow\bigsqcup_{i=1}^{n}\mathcal{U}X_{i}$ by mapping each $\sigma\in\mathcal{U}X$ to the unique $\sigma_{i}$ such that $\sigma_{i}\in\mathcal{U}X_{i}$. This map is a homeomorphism.

\end{proof}

The following definition is a slight variation of a definition in \cite{thesis} to characterize the covering dimension of the boundaries of coarse proximity spaces. A similar definition is given by Hartmann in \cite{hartmann}.  Both definitions are based on the definition of a particular kind of cover for small scale proximity spaces defined by Smirnov in \cite{Smirnov} to define a dimension function for proximity spaces. We will use the following definition in later sections when defining a coarse analog of a fine collection of canonical covers.

\begin{definition}
A {\bf b-cover} (or {\bf coarse cover with respect to b}) of a coarse proximity space $(X,\mathcal{B},{\bf b})$ is a finite collection of sets $\{A_{1},\ldots,A_{n}\}$ satisfying:\\
\begin{enumerate}
\item There is a unique, potentially empty, bounded $A_{i}$.
\item The collection covers $X$
\item There are set $C_{1},\ldots,C_{n}$ that also cover and satisfy $C_{i}\ll A_{i}$ for each $i$. 
\end{enumerate}
\end{definition}

\begin{proposition}\label{b cover to open cover}
Let $(X,\mathcal{B},{\bf b})$ be a coarse proximity space with boundary $\mathcal{U}X$. If $\{A_{1},\ldots,A_{n}\}$ is a ${\bf b}$-cover of $X$, then $\{int(tr(A_{1})),\ldots,int(tr(A_{n})\}$ is an open cover of $\mathcal{U}X$.

\end{proposition}

\begin{proposition}\label{intermediate covers}
Let $(X,\mathcal{B},{\bf b})$ be a coarse proximity space with boundary $\mathcal{U}X$. If $\mathcal{U}=\{U_{1},\ldots,U_{n}\}$ and $\mathcal{V}=\{V_{1},\ldots,V_{m}\}$ are open covers of $\mathcal{U}X$ such that $\overline{\mathcal{V}}=\{\overline{V}_{1},\ldots,\overline{V}_{m}\}$ refines $\mathcal{U}$, then there is a ${\bf b}$-cover $\{A_{1},\ldots,A_{m}\}$ of $X$ such that $\overline{V}_{i}\subseteq tr(A_{i})$ for each $i$ and $\overline{V}_{i}\subseteq U_{j}$ implies that $tr(A_{i})\subseteq U_{j}$.

\end{proposition}

It is clear from the fact that coarse proximity maps copreserve the coarse neighbourhood operator $\ll$ that the preimages of ${\bf b}$-covers under coarse proximity maps are ${\bf b}$-covers of the domain of the map. Some final definitions and results we will make use of are the following as can be found in \cite{thesis}.

\begin{definition}
Given subsets $A,C$ of a coarse proximity space $(X,\mathcal{B},{\bf b})$ we define
\begin{enumerate}
\item $A\sqsubseteq C$ if for all $D\subseteq X$ we have that $A{\bf b}D$ implies $C{\bf b}D$.
\item $A\ll_{w} C$ if there is a $D\subseteq X$ such that $A\ll D\sqsubseteq C$.
\end{enumerate}

In case $(1)$ we say that $A$ is a {\bf coarse subset} of $C$. In case $(2)$ we say that $C$ is a {\bf weak coarse neighbourhood} of $A$.
\end{definition}

The weak coarse neighbourhood operator is in some ways nicer than the normal coarse neighbourhood operator as can be seen in the following proposition.

\begin{proposition}\label{weak neighbourhood operator}
Let $A,C$ be subsets of coarse proximity spaces $(X,\mathcal{B},{\bf b})$. Then
\begin{enumerate}
\item $A\sqsubseteq C$ if and only if $tr(A)\subseteq tr(C)$.
\item $A\ll_{w}C$ if and only if $tr(A)\subseteq int(tr(C))$.
\end{enumerate}
\end{proposition}

\subsection{Cofinal Dimension}
A large part of our results will be focused on defining an analog of Ponomarev's cofinal dimension (a topological property) for coarse proximity spaces. For the sake of self containment we provide all of the necessary definitions and important results about the cofinal dimension here. They are as found in \cite{pears}. 

\begin{definition}\label{canonical cover definition}
A locally finite closed cover $\{\mathcal{F}_{\alpha}\}_{\alpha\in\Omega}$ of a topological space $X$ is called a {\bf canonical cover} if there is a disjoint collection of open sets $\{U_{\alpha}\}_{\alpha\in\Omega}$ such that $\mathcal{F}_{\alpha}=cl(U_{\alpha})$ for each $\alpha\in\Omega$. 
\end{definition}

\begin{lemma}\label{interior of canonical is disjoint}
If $\{F_{\alpha}\}_{\alpha\in\Omega}$ is a canonical cover of a topological space $X$, then the collection $\{int(F_{\alpha})\}_{\alpha\in\Omega}$ is a disjoint family.

\end{lemma}
\begin{proof}
Let $\{F_{\alpha}\}_{\alpha\in\Omega}$ be a canonical cover of a space $X$ with $\{G_{\alpha}\}_{\alpha\in\Omega}$ a disjoint collection of open sets such that $F_{\alpha}=cl(G_{\alpha})$ for each $\alpha$. Assume towards a contradiction that there are $\alpha\neq\beta$ and $x\in X$ such that $x\in int(F_{\alpha})\cap int(F_{\beta})$. As $cl(G_{\alpha})=F_{\alpha}$ and $cl(G_{\beta})=F_{\beta}$ we have that $int(F_{\alpha})\cap G_{\beta}\neq\emptyset$ and $int(F_{\beta})\cap G_{\alpha}\neq\emptyset$. This implies that $G_{\alpha}$ and $G_{\beta}$ must intersect, which is a contradiction. Thus the family $\{int(F_{\alpha})\}_{\alpha\in\Omega}$ must be disjoint.
\end{proof}

\begin{lemma}\label{compact canonical is finite}
If $\{F_{\alpha}\}_{\alpha\in\Omega}$ is a canonical cover of a compact Hausdorff space $X$, then $\Omega$ must be finite.
\end{lemma}
\begin{proof}
By Lemma \ref{interior of canonical is disjoint} we have that $\{int(F_{\alpha})\}_{\alpha\in\Omega}$ is a disjoint family. For each $\alpha$ let $x_{\alpha}\in int(F_{\alpha})$ be given and $W_{\alpha}\subseteq int(F_{\alpha})$ be an open set such that $cl(W_{\alpha})\subseteq int(F_{\alpha})$. Then define $U=X\setminus\left(\bigcup_{\alpha}cl(W_{\alpha})\right)$. Then the collection $\{U\}\cup\{int(F_{\alpha})\mid\alpha\in\Omega\}$ is an open cover of $X$ which must admit a finite subcover by compactness. As each $int(F_{\alpha})$ contains some point not contained in any other element of the open cover, we have that the finite subcover must contain each $int(F_{\alpha})$. Therefore $\Omega$ must be finite.

\end{proof}

\begin{proposition}\label{closed covers admit canonical refinements}
If $\{U_{\alpha}\}_{\alpha\in\Omega}$ is an open cover of a topological space $X$ whose corresponding closed cover $\{cl(U_{\alpha})\}_{\alpha\in\Omega}$ is locally finite, then there is a canonical cover $\{F_{\lambda}\}_{\lambda\in\Lambda}$ of $X$ that refines $\{cl(U_{\alpha})\}_{\alpha\in\Omega}$. 
\end{proposition}

\begin{definition}\label{fine system of covers}
A collection $\{\mathcal{F}_{\alpha}\}_{\alpha\in\Omega}$ of canonical covers of a topological space $X$ is called {\bf fine} if every open cover $\mathcal{U}$ of $X$ is refined by some $\mathcal{F}_{\alpha}$. 
\end{definition}

\begin{definition}\label{perfectly zero-dimensional}
A topological space $X$ with a basis of open and closed sets such that each open covering $X$ has a disjoint open refinement is called a {\bf perfectly zero-dimensional space}.
\end{definition}

\begin{definition}
A collection of canonical covers $\{\mathcal{F}_{\alpha}\}_{\alpha\in\Omega}$ of a topological space $X$ is called {\bf directed} if $\Omega$ is a directed set and $\alpha\leq\beta$ implies $\mathcal{F}_{\beta}$ refines $\mathcal{F}_{\alpha}$, where $\alpha,\beta\in\Omega$.
\end{definition}

\begin{definition}
Given a topological space $X$, the {\bf cofinal dimension} of $X$, denoted $\Delta(X)$, is defined to be $-1$ if $X$ is empty. Otherwise $\Delta(X)$ is the least nonnegative integer $n$ such that $X$ admits a fine directed collection of canonical covers $\{F_{\alpha}\}_{\alpha\in\Omega}$ such that the order of $\mathcal{F}_{\alpha}$ does not exceed $n+1$ for each $\alpha\in\Omega$. If no such $n$ exists, we say that $\Delta(X)$ is infinite.
\end{definition}

\begin{proposition}\label{cofinal dimension zero}
For a topological space $X$, the following are equivalent:
\begin{enumerate}
\item $X$ is perfectly zero dimensional,
\item $X$ is a paracompact regular space such that $dim(X)=0$,
\item $X$ is a regular space such that $\Delta(X)=0$.
\end{enumerate}
\end{proposition}

\begin{definition}
A continuous surjection $f:X\rightarrow Y$ between topological spaces is called {\bf irreducible} if $f(A)\neq Y$ for each proper closed set $A$ of $X$.
\end{definition}

\noindent
As the following proposition from \cite{pears} shows, irreducible maps are easy to find.

\begin{proposition}\label{irreducible from perfect}
If $f:X\rightarrow Y$ is a perfect map, then there is a closed set $A\subseteq X$ such that $f\mid A:A\rightarrow Y$ is a perfect irreducible mapping.
\end{proposition}

In particular, every surjection between compact Hausdorff spaces can be restricted to a subset of its domain on which it is irreducible.
\vspace{\baselineskip}

\noindent

Given a surjection $f:X\rightarrow Y$ and subset $B\subseteq X$ define

\[f^{*}(B)=\{y\in Y\mid f^{-1}(y)\subseteq B\}\]

These sets can be used to characterize irreducible maps as well as closed maps.

\begin{lemma}\label{star of open is open}
A continuous surjection $f:X\rightarrow Y$ is an irreducible map if and only if $f^{*}(U)$ is nonempty for each nonempty open set $U\subseteq X$. The map $f$ is closed if and only if $f^{*}(U)$ is open for each nonempty open set $U\subseteq X$.

\end{lemma}

\begin{lemma}\label{commute irreducible}
If $f:X\rightarrow Y$ is a closed irreducible map, then for each open $U\subseteq X$ we have that 

\[f(cl(U))=cl(f^{*}(U))\]

\end{lemma}

\begin{proposition}\label{canonical to canonical}
If $f:X\rightarrow Y$ is a perfect irreducible map and $\{F_{\alpha}\}_{\alpha\in\Omega}$ is a canonical covering of $X$, then $\{f(F_{\alpha})\}_{\alpha\in\Omega}$ is a canonical covering of $Y$. 
\end{proposition}

\begin{theorem}\label{cofinal dimension characterization}
For a regular topological space $X$, the following are equivalent:
\begin{enumerate}
\item $\Delta(X)\leq n$
\item there is a perfectly zero dimensional space $Z$ and a continuous closed surjection $f:Z\rightarrow X$ such that $f^{-1}(x)$ contains at most $n+1$ points for each $x\in X$. 
\end{enumerate}
\end{theorem}

The proof of the above theorem in \cite{pears} constructs a perfectly zero dimensional space in a particular way. This construction will be mimicked in section $8$ so we will review it here. We will omit the verification of the various details of the construction, but they can be found in \cite{pears}.
\vspace{\baselineskip}

\noindent
Let $X$ be a regular topological space and $\mathcal{F}=\{\mathcal{F}^{\alpha}\}_{\alpha\in\Omega}$ be a fine directed set of canonical covers of $X$ such that the order of each $\mathcal{F}^{\alpha}$ does not exceed $n+1$.  Assume that for each $\alpha$ we have that $\mathcal{F}^{\alpha}=\{F^{\alpha}_{\lambda}\}_{\lambda\in\Lambda(\alpha)}$ where $\Lambda(\alpha)$ is an index set for $\mathcal{F}_{\alpha}$. The family $\mathcal{F}$ can be {\bf strongly directed} in that for $\alpha,\beta\in\Omega$ such that $\beta\geq\alpha$ there is a function $\rho_{\beta\alpha}:\Lambda(\beta)\rightarrow\Lambda(\alpha)$ such that $\{\Lambda(\alpha),\rho_{\beta\alpha}\}_{\alpha,\beta\in\Omega}$ is an inverse system of sets over $\Omega$ and for each $\lambda\in\Lambda(\alpha)$ we have

\[F^{\alpha}_{\lambda}=\bigcup\{F^{\beta}_{\mu}\mid \rho_{\beta\alpha}(\mu)=\lambda\}\]

For each $\alpha\in\Omega$ a topological space $X_{\alpha}$ can be constructed by defining $X_{\alpha}=\bigcup_{\lambda\in\Lambda(\alpha)}F^{\alpha}_{\lambda}\times\{\lambda\}$, which is the disjoint union of the elements of $\mathcal{F}^{\alpha}$ and giving it the corresponding disjoint union topology. The obvious function $\sigma_{\alpha}:X_{\alpha}\rightarrow X$ defined by $\sigma_{\alpha}(x,\lambda)=x$ is continuous. When $\beta,\alpha\in\Omega$ are such that $\beta\geq\alpha$ the function $g_{\beta\alpha}:X_{\beta}\rightarrow X_{\alpha}$ defined by $g_{\beta\alpha}(x,\lambda)=(x,\rho_{\beta\alpha}(\lambda))$ is a continuous closed surjection. These maps make $\{X_{\alpha},g_{\beta\alpha}\}_{\alpha,\beta\in\Omega}$ an inverse system of topological spaces over $\Omega$. The space $Z$ used in \cite{pears} to prove Theorem \ref{cofinal dimension characterization} is the inverse limit of this system. The elements of $Z$ are of the form $\{(x,\lambda,\alpha)\}_{\alpha\in\Omega}$ where $\lambda\in\Lambda(\alpha)$. The obvious function that maps onto the first coordinate is the irreducible closed $n+1$-to-$1$ surjection mentioned in the theorem. Calling this function $h$ and denoting the projections from $Z$ to $X_{\alpha}$ by $\pi_{\alpha}$ we have that $\sigma_{\alpha}\circ\pi_{\alpha}=h$ for each $\alpha$. 
\vspace{\baselineskip}

\noindent
We will finish this section with a minor original result. Specifically we will show that, for compact Hausdorff spaces, the definition of canonical covers can be weakened without changing the cofinal dimension.

\begin{definition}\label{almost canonical}
Let $X$ be a topological space. We say that a locally finite closed cover $\{\mathcal{F}_{\alpha}\}_{\alpha\in\Omega}$ is an {\bf almost canonical cover} if it satisfies the following:\\
\begin{enumerate}
\item $int(F_{\alpha})\neq\emptyset$ for all $\alpha\in\Omega$
\item For $\alpha\neq\beta$, $int(F_{\alpha})\cap int(\mathcal{F}_{\beta})=\emptyset$
\item If $x\in X$ and $U\subseteq X$ is an open subset of $X$ that contains $x$, then there is a $y\in U$ and an $\alpha\in\Omega$ such that $y\in int(\mathcal{F}_{\alpha})$.
\end{enumerate}

\end{definition}

\begin{proposition}\label{almost canonical is equivalent}
Let $X$ be a compact Hausdorff space. Then the following are equivalent:\\
\begin{enumerate}
\item $\Delta(X)\leq n$
\item The statement ``$\Delta(X)\leq n$" still holds upon replacing every instance of ``canonical cover" in the definition of $\Delta(X)$ with "almost canonical cover".
\end{enumerate}
\end{proposition}
\begin{proof}
That $(1)$ implies $(2)$ is immediate upon observing that every canonical cover of a compact Hausdorff space $X$ is an almost canonical cover. 
\vspace{\baselineskip}

To prove that $(2)$ implies $(1)$ we first show that every almost canonical cover $\mathcal{V}=\{\mathcal{F}_{\alpha}\}_{\alpha\in\Omega}$ induces a canonical cover $\mathcal{U}$ that refines $\mathcal{V}$ and has order not exceeding that of $\mathcal{V}$. If $\mathcal{V}=\{\mathcal{F}_{\alpha}\}_{\alpha\in\Omega}$ is an almost canonical cover define $\mathcal{U}=\{\hat{\mathcal{F}}_{\alpha}=cl(int(\mathcal{F}_{\alpha}))\mid \alpha\in\Omega\}$. We claim that $\mathcal{U}$ is our desired cover. First, note that because $\mathcal{V}$ is locally finite, so is $\mathcal{U}$. Secondly, let $x\in X$ be given. As $\mathcal{U}$ is locally finite, and $\mathcal{V}$ satisfies $(2)$ of Definition \ref{almost canonical} there is a finite subcollection $\{\hat{\mathcal{F}}_{\alpha_{1}},\ldots,\hat{\mathcal{F}}_{\alpha_{m}}\}$ of $\mathcal{V}$ such that $x$ admits a neighbourhood basis that only intersect the $\hat{\mathcal{F}}_{\alpha_{i}}$. By the pigeonhole principle $x$ must be an element of one of the $\hat{\mathcal{F}}_{\alpha_{i}}$. Therefore $\mathcal{U}$ covers $X$. Finally, as the interiors of the $\mathcal{F}_{\alpha}$ are disjoint, so are the interiors of the $\hat{\mathcal{F}}_{\alpha}$. We then have that $\mathcal{U}$ is a canonical cover of $X$. It is clear from the definition of $\mathcal{U}$ that the order of $\mathcal{U}$ does not exceed that of $\mathcal{V}$. Now, assume that we have that $X$ admits a directed fine family $\{\mathcal{V}_{\alpha}\}_{\alpha\in\Lambda}$ of almost canonical covers none of which has order exceeding $n$. Then the corresponding collection of canonical covers $\{\mathcal{U}_{\alpha}\}_{\alpha\in\Lambda}$ is a fine directed collection of canonical covers none of which has order exceeding $n$. Therefore $\Delta(X)\leq n$. 

\end{proof}

\begin{lemma}\label{image of almost canonical is almost canonical}
Let $f:X\rightarrow Y$ be a perfect irreducible map between compact Hausdorff spaces. If $\mathcal{F}=\{A_{1},\ldots,A_{n}\}$ is a almost canonical cover of $X$, then $f(\mathcal{F})=\{f(A_{1}),\ldots, f(A_{n})\}$ is a almost canonical cover of $Y$. 
\end{lemma}
\begin{proof}
Let $\{A_{1},\ldots,A_{n}\}$ be an almost canonical cover of $X$. For each $i$ define $F_{i}=cl_{X}(int(A_{i}))$. As shown in the proof of Proposition \ref{almost canonical is equivalent}, the set $\{F_{1},\ldots,F_{n}\}$ is a canonical cover of $X$. By Proposition \ref{canonical to canonical} the set $\{f(F_{1}),\ldots,f(F_{n})\}$ is a canonical cover of $Y$. To show that $\{f(A_{1}),\ldots,f(A_{n})\}$ we will begin by showing that if $y\in int(f(A_{i}))$, then $y\in int(f(F_{i}))$. Let $W\subseteq Y$ be an open set containing $y$ such that $W\subseteq int(f(A_{i}))$. Define $W_{i}=f^{-1}(W)\cap A_{i}$. As $W\subseteq int(f(A_{i}))$ we have that $f(W_{i})=W$. We claim that $W_{i}\subseteq int(A_{i})$. If $x\in W_{i}$ is not in $int(A_{i})$, then because $f^{-1}(W)$ is open in $X$ and $\mathcal{F}$ is an almost canonical cover of $X$, there is an open set $V_{j}$ contained in $f^{-1}(W)$ that is completely contained in $int(A_{j})$ for some $j\neq i$. Then $f(V_{j})\subseteq W$. Note that, by Lemma \ref{commute irreducible} we have that $f(F_{j})=f(cl(int(A_{j})))=cl(f^{*}(int(A_{j})))$. As $V_{j}\subseteq int(A_{j})$ we have that $f^{*}(V_{j})$ is a nonempty open subset of $f(F_{j})$ by Lemma \ref{commute irreducible}. However this is impossible as $f(V_{j})\subseteq W\subseteq int(F_{i})$ which is disjoint from $int(F_{j})$. Therefore $W_{i}\subseteq int(A_{i})$. Consequently $W_{i}$ is open as well. Now, $f(cl(W_{i}))\subseteq f(cl(int(A_{i})))=f(F_{i})$ which gives us that $y\in f(F_{i})$. Finally, as $f(W_{i})=W$ which is open in $Y$ we have that $y\in int(f(F_{i}))$ and $int(f(A_{i}))=int(f(F_{i}))$ which gives us that the interiors of the various $f(A_{i})$ are disjoint.
\vspace{\baselineskip}

Finally, we need to show that if $y\in Y$ is any point and $W\subseteq Y$ is an open subset of $Y$ containing $y$, then there is an $i$ and an $x\in W$ such that $x\in W\cap int(f(A_{i}))$. If $y$ and $W$ are given we may consider $f^{-1}(y)$ and $f^{-1}(W)$. As the collection of $A_{j}$'s is almost canonical there must be some $i$ for which $f^{-1}(W)\cap int(A_{i})$ is nontrivial. Call this intersection $W_{i}$. By Proposition \ref{star of open is open} we have that $f^{*}(W_{i})$ is open in $Y$. This set is necessarily contained in $W$. Then if $z\in f^{*}(W_{i})$ we have that $f^{-1}(z)\subseteq f^{-1}(f^{*}(W_{i}))\subseteq int(A_{i})$. Then $f(f^{-1}(f^{*}(W)))$ is an open set containing $z$, contained in $W$, that is also contained in $f(A_{i})$.

\end{proof}

\section{Asymptotic dimension of coarse proximity spaces}\label{asymptotic dimension of coarse proximity spaces}

In this section we will define and prove some very basic properties of the asymptotic dimension of coarse proximity spaces. The results and definitions are largely rephrasings of definitions and results for metric spaces and coarse spaces as one might find in \cite{Bell} and \cite{Roe}. We also provide some results about uniformly bounded covers in coarse proximity spaces that will be used in later sections. The definition of uniformly bounded covers that we present here is identical in form to the definition of uniformly bounded covers for asymptotic resemblance spaces as in \cite{Honari}.

\begin{definition}
	Let $(X,\mathcal{B}, {\bf b})$ be a coarse proximity space. A collection, $\mathcal{U}$, of subsets of $X$ is called a {\bf uniformly bounded family} if
	\begin{enumerate}
		\item $B\in\mathcal{U}$ implies $B\in\mathcal{B}$
		\item For all $A,C\subseteq X$, $A\subseteq st(C,\mathcal{U})$ and $C\subseteq st(C,\mathcal{U})$ implies $A\phi C$
	\end{enumerate}

If in addition we have that for all $x\in X$ there is a $B\in\mathcal{U}$ such that $x\in B$ then we call $\mathcal{U}$ a {\bf uniformly bounded cover} of $X$.
\end{definition}

\begin{proposition}\label{3.3}
	Let $(X,d)$ be a metric space and $(X,\mathcal{B},{\bf b})$ the corresponding metric coarse proximity structure. Then a collection $\mathcal{U}$ of subsets of $X$ is uniformly bounded in metric $d$ if and only if it is a uniformly bounded family with respect to the coarse proximity structure. Consequently $\mathcal{U}$ is a uniformly bounded cover of $X$ in the metric sense if and only if it is so in the coarse proximity sense.
\end{proposition}
\begin{proof}
Let $\mathcal{U}$ be a collection of subsets of $X$ that is uniformly bounded with respect to the metric $d$. Then there is an $M<\infty$ such that for all $B\in\mathcal{U}$ we have that $diam(B)\leq M$. Thus $B\in\mathcal{B}$ for all $B\in\mathcal{U}$. If $A,C\subseteq X$ are such that $A\subseteq st(C,\mathcal{U})$ and $C\subseteq st(A,\mathcal{U})$ then the Hausdorff distance between $A$ and $C$ is at most $M$, which implies $A\phi C$. Therefore $\mathcal{U}$ is a uniformly bounded family with respect to the coarse proximity structure.
\vspace{\baselineskip}

Now assume that $\mathcal{U}$ is a uniformly bounded family of $X$ with respect to the coarse proximity structure. The second condition in the definition of a uniformly bounded family says that if $A,C\subseteq X$ are such that $A\subseteq st(C,\mathcal{U})$ and $C\subseteq st(A,\mathcal{U})$ then $A\phi C$ where $\phi$ is the relation defined by having finite Hausdorff distance. If $B\in\mathcal{U}$ is unbounded in the metric sense then for any $x\in B$ we have that $B\subseteq st(x,\mathcal{U})$ and $\{x\}\subseteq st(B,\mathcal{U})$, but $\{x\}\bar{\phi}B$. Therefore every element of $\mathcal{U}$ must be bounded in the metric sense. If $\mathcal{U}$ is not uniformly bounded in the metric sense, then for every $n\in\mathcal{N}$ we can find a $B_{n}\in\mathcal{U}$ such that $diam(B_{n})\geq n$. Fix a point $x_{0}\in X$.

\begin{claim}
For every $n\in\mathbb{N}$ there is an element $C_{n}\in\mathcal{U}$ such that $diam(C_{n}\setminus B(x_{0},n^{2}))>n$.

\end{claim}
\begin{proof}
Assume otherwise. Then there is an $n\in\mathbb{N}$ such that if $B\in\mathcal{U}$, then $diam(C\setminus B(x_{0},n^{2}))\leq n$. Then by the triangle inequality we have that $diam(B)\leq n+2n^{2}$ for all $B\in\mathcal{U}$, contradicting our assumption that $\mathcal{U}$ is not uniformly bounded in the metric sense. Thus the claim holds.
\end{proof}

Now, let $B_{1}\in\mathcal{U}$ be such that $diam(B_{1})>1$ and let $x_{1},y_{1}\in B_{1}$. Choose $m$ large enough so that $B(\{x_{1},y_{1}\},n^{2})\subseteq B(x_{0},m^{2})$ and let $B_{2}\in\mathcal{U}$ be such that $diam(B_{2}\setminus B(x_{0},m^{2}))>m$ as given by the claim. Let $x_{2},y_{2}\in B_{2}\setminus B(x_{0},m^{2})$ be such that $d(x_{2},y_{2})>m$. Proceeding in like fashion we may by induction find sequences $\{x_{n}\}_{n\in\mathbb{N}}$ and $\{y_{n}\}_{n\in\mathbb{N}}$ with the following properties:
\begin{enumerate}
\item $d(x_{j+1},\{x_{1},\ldots,x_{j}\})>j^{2}$ for $1\leq j\leq n-1$
\item $d(y_{j+1},\{y_{1},\ldots,y_{j}\})>j^{2}$ for $1\leq j\leq n-1$
\item $d(x_{j},y_{j})>j$ for $1\leq j\leq n$
\item $d(y_{j+1},\{x_{1},\ldots,x_{j}\})>j^{2}$ for $1\leq j\leq n-1$
\item $d(x_{j+1},\{y_{1},\ldots,y_{j}\})>j^{2}$ for $1\leq j\leq n-1$
\item For $1\leq j\leq n$ there is a $B_{j}\in\mathcal{U}$ such that $x_{j},y_{j}\in B_{j}$.
\end{enumerate}
 Define $A$ to be the collection of points $x_{n}$ and $C$ to be the collection of points $y_{n}$ it is clear from construction that $A\subseteq st(C,\mathcal{U})$ and $C\subseteq st(A,\mathcal{U})$ but $A$ and $C$ do not have finite Hausdorff distance to each other. This contradicts $\mathcal{U}$ being uniformly bounded with respect to the metric coarse proximity structure. Therefore we must have that $\mathcal{U}$ is uniformly bounded in the metric sense. The statement about uniformly bounded covers then follows trivially.
\end{proof}

\begin{definition}\label{asymptotic dimension for coarse proximity spaces}
Given a coarse proximity space $(X,\mathcal{B},{\bf b})$, the {\bf asymptotic dimension} of $X$, denoted $asdim(X)$, is the least nonnegative integer $n$ such that for every uniformly bounded cover $\mathcal{U}$ of $X$, there is a uniformly bounded cover $\mathcal{V}$ of order at most $n+1$ such that $\mathcal{U}<\mathcal{V}$.
\end{definition}

This notion of asymptotic dimension agrees with the usual definition on the metric coarse proximity structure. Arguably the greatest extant result connecting the asymptotic dimension of metric spaces to dimensional invariants of the Higson corona is the following result prove in \cite{dranishnikov} and \cite{asdimlessthandim}.

\begin{theorem}\label{asdim and covering dimension}
If $(X,d)$ is an unbounded proper metric space with $asdim(X)<\infty$, then $asdim(X)=dim(\nu X)$.
\end{theorem}

\begin{lemma}
	If $\mathcal{U}$ and $\mathcal{V}$ are uniformly bounded covers of a coarse proximity space $(X,\mathcal{B},{\bf b})$ then $st(\mathcal{U},\mathcal{V})$ is a uniformly bounded cover of $X$.
\end{lemma}
\begin{proof}

If $W\in st(\mathcal{U},\mathcal{V})$ then $W=st(U,\mathcal{V})$ for some $U\in\mathcal{U}$. Elements of $\mathcal{U}$ are bounded and $U\phi st(U,\mathcal{V})$, so $W$ must also be bounded. Recall that if $\mathcal{W}$ is a uniformly bounded cover of $X$, then $st(A,\mathcal{W})\phi A$ for all $A\subseteq X$. Then note that 

\[st(A,st(\mathcal{U},\mathcal{V}))=st(st(st(A,\mathcal{V}),\mathcal{U}),\mathcal{V})\cup st(st(A,\mathcal{U}),\mathcal{V})\]

Which yields $A\phi st(A,st(\mathcal{U},\mathcal{V}))$. Now assume $A,C\subseteq X$ are such that $A\subseteq st(C,st(\mathcal{U},\mathcal{V}))$ and $C\subseteq st(A,st(\mathcal{U},\mathcal{V}))$ then if $\hat{C}\subseteq C$ is an unbounded set then $\hat{C}\subseteq st(A,st(\mathcal{U},\mathcal{V}))$ and because $st(A,st(\mathcal{U},\mathcal{V}))\phi A$ we have that $A{\bf b}\hat{C}$. Similarly if $\hat{A}\subseteq A$ is unbounded we have $\hat{A}{\bf b}C$. Therefore $A\phi C$ and $st(\mathcal{U},\mathcal{V})$ is a uniformly bounded cover. 
\end{proof}

\begin{example}
	Let $\mathbb{N}$ be equipped the metric bornology and the coarse proximity ${\bf b}$ induced by the one point compactification of $\mathbb{N}$. That is, sets are considered bounded if and only if they are metrically bounded and two sets $A,C\subseteq\mathbb{N}$ are coarsely close if and only if they are unbounded. Then the boundary $\mathcal{U}\mathbb{N}$ is a single point. Define $\mathcal{U}=\{[n,n+1]\mid n\in\mathbb{N}\}$. Then $\mathcal{U}$ is a uniformly bounded cover of $\mathbb{N}$. If $\mathcal{V}$ is any uniformly bounded cover of $\mathbb{N}$ that is refined by $\mathcal{U}$ then $\mathcal{V}$ must have order at least $2$, which implies $asdim(\mathbb{N})\geq 1$ with this coarse proximity structure. 
\end{example}

\begin{definition}
Let $(X,\mathcal{B},{\bf b})$ be a coarse proximity space, $\mathcal{U}$ a collection of subsets of $X$, and $\mathcal{V}$ a uniformly bounded cover of $X$. The $\mathcal{V}$-{\bf multiplicity} of $\mathcal{U}$ is the greatest integer $n$ such that there is a $V\in\mathcal{V}$ and $U_{1},\ldots,U_{n}\in\mathcal{U}$ such that $U_{i}\cap V\neq\emptyset$ for $1\leq i\leq n$. 
\end{definition}

\begin{proposition}\label{asdim 0 characterization}
Let $(X,\mathcal{B},{\bf b})$ be a coarse proximity space. Then the following are equivalent:\\
\begin{enumerate}
\item $asdim(X)= 0$
\item For every uniformly bounded cover $\mathcal{V}$ there is a uniformly bounded cover $\mathcal{U}$ such that the $\mathcal{V}$-multiplicity of $\mathcal{U}$ is at most $1$. 
\item For every uniformly bounded cover $\mathcal{U}$, the set of $\mathcal{U}$-components in $X$ is uniformly bounded.
\end{enumerate}
\end{proposition}
\begin{proof}
$(1)\implies(2)$ Assume that $asdim(X)=0$ and let $\mathcal{V}$ be a uniformly bounded cover. Then there is a uniformly bounded cover $\mathcal{U}$ of order at most $1$ that is refined by $st(st(\mathcal{V},\mathcal{V}),\mathcal{V})$. Then the $\mathcal{V}$-multiplicity of $\mathcal{U}$ is at most $1$. 
\vspace{\baselineskip}

$(2)\implies(1)$ Assume $(2)$ and let $\mathcal{V}$ be a uniformly bounded cover. Then there is a uniformly bounded cover $\mathcal{U}$ of $\mathcal{V}$-multiplicity at most $1$. Then $\mathcal{V}$ must refine $\mathcal{U}$, so $asdim(X)=0$.
\vspace{\baselineskip}

$(1)\implies (3)$ Assume that $asdim(X)=0$, that $\mathcal{U}$ is a uniformly bounded cover, and that $\mathcal{V}$ is a uniformly bounded cover with $\mathcal{U}$-multiplicity $1$ (as given by $(2)$). Then $\mathcal{U}$ refines $\mathcal{V}$ and the set of $\mathcal{U}$-components refines $\mathcal{V}$ and is therefore uniformly bounded.
\vspace{\baselineskip}

$(3)\implies (1)$ Let $\mathcal{U}$ be a uniformly bounded cover of $X$. Then the collection $\mathcal{V}$ of $\mathcal{U}$-components is uniformly bounded and is refined by $\mathcal{U}$. The order of $\mathcal{V}$ is $1$, so $asdim(X)=0$.
\end{proof}

\begin{definition}
Given a coarse proximity space $(X,\mathcal{B},{\bf b})$ we say that the bornology $\mathcal{B}$ is {\bf generated} by a collection $\{B_{\alpha}\}_{\alpha\in\Omega}\subseteq\mathcal{B}$ if for all $B\in\mathcal{B}$ there is a $B_{\alpha}\in\{B_{\alpha}\}_{\alpha\in\Omega}$ such that $B\subseteq B_{\alpha}$. Similarly, if $\{\mathcal{U}_{\alpha}\}_{\alpha\in\Omega}$ is a collection of uniformly bounded covers of $X$ we say that the collection {\bf generates} the uniformly bounded covers of $X$ if for every uniformly bounded cover $\mathcal{U}$ there is some $\mathcal{U}_{\alpha}$ such that $\mathcal{U}$ refines $\mathcal{U}_{\alpha}$. 

\end{definition}

\begin{definition}
A coarse proximity space $(X,\mathcal{B},{\bf b})$ is said to be {\bf witnessed by uniform covers} if for all $A,C\subseteq X$, $A{\bf b}C$ implies that there is a uniformly bounded cover $\mathcal{U}$ of $X$ such that $st(A,\mathcal{U})\cap C$ is unbounded. In general if for a particular pair of sets $A,C\subseteq X$ we have that $A{\bf b}C$ and there is a uniformly bounded cover $\mathcal{U}$ such that $st(A,\mathcal{U})\cap C$ is unbounded, then we say that $\mathcal{U}$ {\bf witnesses} $A$ and $B$ being close. If $\mathcal{C}$ is a family of uniformly bounded covers of $X$ such that $A{\bf b}C$ implies that there is some $\mathcal{V}\in\mathcal{C}$ such that $A\cap st(C,\mathcal{V})$ is unbounded then we say that the coarse proximity is witnessed by the family $\mathcal{C}$.
\end{definition}

\begin{remark}
An equivalent condition for a coarse proximity to be witnessed by uniform covers if that $A{\bf b}C$ implies that there is a uniformly bounded cover $\mathcal{U}$ for which $st(A,\mathcal{U})\cap st(C,\mathcal{U})$ is unbounded.
\end{remark}

\begin{lemma}\label{countable generation criterion}
Let $(X,\mathcal{B},{\bf b})$ be an infinite coarse proximity space in which $\mathcal{B}$ is countably generated and let $\mathcal{C}=\{\mathcal{V}_{n}\}_{n\in\mathbb{N}}$ be a family of uniformly bounded covers that witnesses the coarse proximity structure on $X$. Then the family of uniformly bounded covers of $X$ is countably generated.
\end{lemma}
\begin{proof}
Say $\mathcal{B}$ is generated by the family $\{B_{n}\}_{n\in\mathbb{N}}$. We then define $\mathcal{W}=\{\mathcal{V}_{n}\cup\{B_{m}\}\mid n,m\in\mathbb{N}\}$. Then $\mathcal{W}$ is a countable set. Enumerate $\mathcal{W}$ by denoting the elements as $\mathcal{U}_{1},\mathcal{U}_{2},\ldots$. We then define $\hat{\mathcal{U}}_{1}=\mathcal{U}_{1}$ and for general $n\geq 2$ we define $\hat{\mathcal{U}}_{n}=st(\hat{\mathcal{U}}_{n-1},\mathcal{U}_{n})$. Let $\hat{\mathcal{W}}$ be the set of all $\hat{\mathcal{U}}_{n}$. We claim that every uniformly bounded cover of $X$ refines an element of $\hat{\mathcal{W}}$. Define a function $d:X\times X\rightarrow[0,\infty)$ by setting $\hat{\mathcal{U}}_{0}$ to be the set of singleton subsets of $X$ and setting $d(x,y)$ to be

\[d(x,y)=\min\{n \mid\exists B\in\hat{\mathcal{U}}_{n},\,\{x,y\}\subseteq B\}\]

As each bounded subset of $X$ is a subset of some element of some $\hat{\mathcal{U}}_{n}$ we have that $d$ is well defined. In fact, it is a metric on $X$ as seen in \cite{Roe}. With this metric we have that the elements of $\hat{\mathcal{U}}_{n}$ have diameter less than or equal to $n$ with respect to this metric. If $\mathcal{V}$ is a uniformly bounded family of $X$ that does not refine any $\hat{\mathcal{U}}_{n}$ then there is a sequence of pairs $(x_{n},y_{n})$ in $X$ such that $\{\{x_{n},y_{n}\}\}_{n\in\mathbb{N}}$ refines $\mathcal{V}$ and $d(x_{n},y_{n})>n$ for each $n\in\mathbb{N}$. As we saw in the proof of Proposition \ref{3.3} we may assume that  the sequence $(d(\{x_{n},y_{n}\},\{x_{n+1},y_{n+1}\}))_{n\in\mathbb{N}}$ diverges to infinity. Defining $A$ to be the collection of $x_{n}$'s and $C$ to be the collection of $y_{n}$'s we have that $A$ and $C$ must be unbounded as if they were both bounded then there would be a $\hat{\mathcal{U}}_{n}$ that contains the union of both of them. Similarly because $A\subseteq st(C,\mathcal{V})$ and $C\subseteq st(A,\mathcal{V})$ we have that $A\phi C$, so both must either unbounded or bounded. Because $A\phi C$ we have that $A{\bf b}C$. Because the coarse proximity on $X$ is witnessed by the family $\{\hat{\mathcal{U}}_{n}\}_{n\in\mathbb{N}}$ we must have that there is some $n$ for which $A\cap st(C,\hat{\mathcal{U}}_{n})$ is unbounded. However, there are only finitely many elements $(x,y)\in A\times C$ for which $d(x,y)\leq n$, so this is impossible. This is a contradiction. Therefore $\mathcal{V}$ must refine some $\hat{\mathcal{U}}_{n}$ meaning that the uniformly bounded covers of $X$ are countably generated.
\end{proof}

\begin{lemma}\label{witnessing weak asr}
If $(X,\mathcal{B},{\bf b})$ is a coarse proximity space whose bornology is countably generated and whose coarse proximity is witnessed by the countably family $\mathcal{C}=\{\mathcal{V}_{n}\}_{n\in\mathbb{N}}$, then there is a countably family of uniformly bounded covers of $X$, $\hat{\mathcal{C}}=\{\hat{\mathcal{V}}_{n}\}_{n\in\mathbb{N}}$ that is linearly ordered by refinement, is such that if $\mathcal{U}$ is a uniformly bounded cover of $X$ then $\mathcal{U}$ refines some $\hat{\mathcal{V}}_{n}$, and satisfies the property that if $A,C\subseteq X$ are such that $A\phi C$ then there is a $\hat{\mathcal{V}}_{n}$ such that $A\subseteq st(C,\hat{\mathcal{V}})$ and $C\subseteq st(A,\hat{\mathcal{V}})$. 
\end{lemma}
\begin{proof}
If $\mathcal{C}=\{\mathcal{V}_{n}\}_{n\in\mathbb{N}}$ is a countable family of uniformly bounded covers of a coarse proximity space $X$, then by the proof of Lemma \ref{countable generation criterion} we have that the linearly ordered family $\hat{\mathcal{C}}=\{\hat{\mathcal{V}}_{n}\}_{n\in\mathbb{N}}$ as described in the statement of the Lemma, exists. Now let $A,C\subseteq X$ be unbounded sets such that $A\phi C$. Let $d$ be the metric on $X$ constructed from the family $\hat{\mathcal{C}}$ as in the proof of Lemma \ref{countable generation criterion}. Assume towards a contradiction that there is no $n\in\mathbb{N}$ such that $A\subseteq st(C,\hat{\mathcal{V}}_{n})$ and $C\subseteq st(A,\hat{\mathcal{V}}_{n})$. This is equivalent to the Hausdorff distance between $A$ and $C$, with respect to the metric $d$, being infinite. By the pigeonhole principle we then have that, without loss of generality, there is an unbounded subset $\hat{A}\subseteq A$ such that for every $n$, $\hat{A}\cap st(C,\hat{\mathcal{U}}_{n})$ is finite. However, because $A\phi C$ we must have that $\hat{A}{\bf b}C$ and because the coarse proximity on $X$ it witnessed by the family $\hat{\mathcal{C}}$ there must be some $n\in\mathbb{N}$ for which $\hat{A}\cap st(C,\hat{\mathcal{U}}_{n})$ is unbounded and therefore infinite. This is a contradiction. Therefore there must be some $n\in\mathbb{N}$ for which $A\subseteq st(C,\hat{\mathcal{U}}_{n})$ and $C\subseteq st(A,\hat{\mathcal{U}}_{n})$. 
\end{proof}

\begin{proposition}\label{witnessability through isomorphisms}
If $f:(X,\mathcal{B}_{X},{\bf b}_{X})\rightarrow (Y,\mathcal{B}_{Y},{\bf b}_{Y})$ is an injective coarse proximity embedding, $A,C\subseteq X$ are such that $A{\bf b}_{X} C$, and there is a uniformly bounded cover $\mathcal{U}$ in $Y$ that witnesses $f(A)$ being close to $f(C)$, then $f^{-1}(\mathcal{U})=\{f^{-1}(B)\mid B\in\mathcal{U}\}$ is a uniformly bounded cover of $X$ that witnesses $A$ being close to $C$. 
\end{proposition}
\begin{proof}
Let $A,C,$ and $\mathcal{U}$ be given. First we will show that $f^{-1}(\mathcal{U})$ is a uniformly bounded cover of $X$. As $\mathcal{U}$ covers $Y$ it is clear that $f^{-1}(\mathcal{U})$ covers $X$. Moreover, as the preimage of bounded sets via a coarse proximity map are bounded we have that all elements of $f^{-1}(\mathcal{U})$ are bounded. Now let $A_{1},A_{2}\subseteq X$ be unbounded sets such that $A_{1}\subseteq st(A_{2},f^{-1}(\mathcal{U}))$ and $A_{2}\subseteq st(A_{1},f^{-1}(\mathcal{U}))$. Then $f(A_{1})\subseteq st(f(A_{2}),\mathcal{U})$ and $f(A_{2})\subseteq st(f(A_{1}),\mathcal{U})$. As $\mathcal{U}$ is a uniformly bounded cover of $Y$ we have that $f(A_{1})\phi_{Y}f(A_{2})$ which implies that $A_{1}\phi_{X}A_{2}$ because $f$ is a coarse proximity embedding. Therefore $f^{-1}(\mathcal{U})$ is a uniformly bounded cover of $X$.
\vspace{\baselineskip}

Now, let $D$ be an unbounded subset of $st(f(A),\mathcal{U})\cap f(C)$. Because $f$ is injective we have that $f^{-1}(D)$ is an unbounded subset of $st(A,f^{-1}(\mathcal{U}))\cap C$, establishing the result.

\end{proof}

\begin{definition}
A subspace $Y$ of a coarse proximity space $(X,\mathcal{B},{\bf b})$ is called {\bf coarsely dense} in $X$ if $Y\phi X$. 
\end{definition}

\begin{proposition}\label{may assume bijective}
Let $(X,d)$ be a metric space equipped with its metric coarse proximity structure, $(Y,\mathcal{B},{\bf b})$ a coarse proximity space, and $f:Y\rightarrow X$ a coarse proximity map such that $f(Y)$ is coarsely dense in $X$. Then there are subspaces $\hat{Y}\subset Y$ and $\hat{X}\subseteq X$ such that $f\mid_{\hat{Y}}$ is a bijective coarse map between $\hat{Y}$ and $\hat{X}$, and $\hat{X}$ is coarsely equivalent to $X$.
\end{proposition}
\begin{proof}
Let $X$, $Y$, and $f$ be given. Define $\hat{X}=f(Y)$. As $f(Y)$ is coarsely dense in $X$ we immediately have that $\hat{X}$ is coarsely equivalent to $X$. For each $z\in\hat{X}$ let $y_{z}\in Y$ be an element such that $f(y_{z})=z$ and define $\hat{Y}$ to be the set of such elements. The result then follows.
\end{proof}

\section{Metrizability of coarse proximity spaces}\label{metrizability}

In this short section we provide a characterization of coarse proximity spaces that are equivalent to some metric space. One may compare the result of this section to the corresponding result about the metrizability of coarse spaces in \cite{Roe}.

\begin{definition}
	A coarse proximity space $(X,\mathcal{B},{\bf b})$ is called {\bf metrizable} if there is a metric $d$ on $X$ such that $\mathcal{B}$ is the bornology of sets that are bounded via the metric $d$ and ${\bf b}$ is the metric coarse proximity induced by $d$.
\end{definition}

\begin{theorem}\label{metrizability criterion}
	Let $(X,\mathcal{B},{\bf b})$ be a coarse proximity space. Then the following are equivalent:
\begin{enumerate}
\item $(X,\mathcal{B},{\bf b})$ is metrizable.
\item $\mathcal{B}$ is countably generated and ${\bf b}$ is witnessed by a countable family of uniformly bounded covers.

\end{enumerate}
\end{theorem}
\begin{proof}
($(1)\Rightarrow (2)$) This direction is trivial.
\vspace{\baselineskip}

($(2)\Leftarrow(1)$) This will follow from Lemma \ref{countable generation criterion} and Lemma \ref{witnessing weak asr}. If $\mathcal{V}=\{\mathcal{V}_{n}\}_{n\in\mathbb{N}}$ is a countable family of uniformly bounded covers that witnesses the coarse proximity ${\bf b}$ and $\mathcal{C}=\{B_{n}\}_{n\in\mathbb{N}}$ is a countable subset of $\mathcal{B}$ that generates the bornology. By Lemma \ref{witnessing weak asr} we may assume that $\mathcal{V}$ is linearly ordered with respect to refinement. We may also assume that every $B_{n}$ is contained in an element of some $\mathcal{V}_{n}$. Defining a metric $d$ as in the proof of Lemma \ref{countable generation criterion} we quickly have that the metric induces the original coarse proximity structure on $X$.
\end{proof}

\section{Inverse limits of coarse proximity spaces}\label{inverse limits}

Let $\Omega$ be a directed set. That is, $\Omega$ is a nonempty set with an order $\leq$ that is reflexive, antisymmetric, transitive, and is such that every two elements have a common upper bound. Let $\{(X_{\alpha},\mathcal{B}_{\alpha},{\bf b}_{\alpha}),f_{\alpha\beta}\}_{\alpha,\beta\in \Omega}$ be an inverse system of coarse proximity spaces over $\Omega$. That is, for each $\alpha\in \Omega$ there is a coarse proximity space $(X_{\alpha},\mathcal{B}_{\alpha},{\bf b}_{\alpha})$, for each pair of distinct elements $\alpha,\beta\in \Omega$ such that $\alpha<\beta$ there is a coarse proximity map $f_{\beta\alpha}:X_{\beta}\rightarrow X_{\alpha}$, and the collection of maps $f_{\beta\alpha}$ is such that if $\alpha<\beta<\gamma$ then $f_{\beta\alpha}\circ f_{\gamma\beta}=f_{\gamma\alpha}$. 
\vspace{\baselineskip}

Consider the product $\prod_{\alpha\in \Omega}X_{\alpha}$. Denote each element of this set by $(x_{\alpha})_{\alpha\in \Omega}$ where $x_{\alpha}\in X_{\alpha}$. Let $Z\subseteq\prod_{\alpha\in \Omega}X_{\alpha}$ be the set of all $(x_{\alpha})_{\alpha\in \Omega}$ such that $f_{\beta\alpha}(x_{\beta})=x_{\alpha}$ for all $\alpha,\beta\in \Omega$ with $\alpha<\beta$. For each $\alpha\in \Omega$ define $\pi_{\alpha}:Z\rightarrow X_{\alpha}$ by $\pi_{\alpha}((x_{\alpha})_{\alpha\in \Omega})=x_{\alpha}$. 

\begin{lemma}\label{inverse limit bornology}
With $Z$ defined as above, define $\mathcal{B}$ to the collection of all subsets $B\subseteq Z$ with the property that there is an $\alpha\in \Omega$ for which $\pi_{\alpha}(B)\in\mathcal{B}_{\alpha}$. Then $\mathcal{B}$ is a bornology on $\mathcal{B}$. 
\end{lemma}
\begin{proof}
Given any point $(x_{\alpha})_{\alpha\in \Omega}\in Z$ and any $\alpha\in \Omega$ we have that $\pi_{\alpha}((x_{\alpha})_{\alpha\in \Omega})=x_{\alpha}$ and the singleton $\{x_{\alpha}\}$ is an element of $\mathcal{B}_{\alpha}$. Therefore, all singleton sets in $Z$ are in $\mathcal{B}$, making $\mathcal{B}$ a cover of $Z$. Similarly, all finite subsets of $Z$ are in $\mathcal{B}$. As images of functions preserve the subset relation we have that $\mathcal{B}$ is closed under subsets. Similarly as functions commute with set theoretic unions we have that $\mathcal{B}$ is closed under finite unions. We then have that $\mathcal{B}$ is a bornology on $Z$. 
\end{proof}

\begin{proposition}
Define $Z$ as above and equip it with the bornology described in Lemma \ref{inverse limit bornology}. Define a binary relation ${\bf b}$ on subsets of $(Z,\mathcal{B})$ by setting $A{\bf b}C$ if and only if $\pi_{\alpha}(A){\bf b}_{\alpha}\pi_{\alpha}(C)$ for all $\alpha\in \Omega$.  If $Z\neq\emptyset$ then the triple $(Z,\mathcal{B},{\bf b})$ is a coarse proximity space.
\end{proposition}
\begin{proof}
Symmetry and the union axiom are clear. Moreover, if $A{\bf b}C$ then it is clear that neither $A$ nor $C$ are in $\mathcal{B}$ as otherwise there is a $\gamma\in \Omega$ such that $\pi_{\gamma}(A)$ of $\pi_{\gamma}(C)$ is in $\mathcal{B}_{\gamma}$ which would imply that the image of $A$ or $C$ under $\pi_{\gamma}$ is coarsely close to a bounded set, which is impossible. As images of unbounded subsets of $Z$ under every $\pi_{\alpha}$ are unbounded we have that if $A,C\subseteq Z$ intersect along an unbounded set, then $A{\bf b}C$. Next, let $A,C\subseteq Z$ be such that $A\bar{\bf b}C$. Then there is a $\gamma\in \Omega$ such that $\pi_{\gamma}(A)\bar{\bf b}_{\gamma}\pi_{\gamma}(C)$. As ${\bf b}_{\gamma}$ is a coarse proximity space there is a $E\subseteq X_{\gamma}$ such that $\pi_{\gamma}(A)\bar{\bf b}_{\gamma}E$ and $(X_{\gamma}\setminus E)\bar{\bf b}_{\gamma}\pi_{\gamma}(C)$. Define $\hat{E}$ to be $\pi_{\gamma}^{-1}(E)$. We then must have that $A\bar{\bf b}\hat{E}$ and $Z\setminus\hat{E}=\pi_{\gamma}^{-1}(X_{\gamma}\setminus E)\bar{\bf b}C$. Finally, let $A,C,D\subseteq Z$ be unbounded sets. It is clear that if $A{\bf b}C$ or $A{\bf b}D$ then $A{\bf b}(C\cup D)$. Conversely, assume that $A{\bf b}(C\cup D)$. If $A\bar{\bf b}C$ and $A\bar{\bf b}D$ then there is a $\gamma\in D$ such that $\pi_{\gamma}(A)\bar{\bf b}_{\gamma}\pi_{\gamma}(C)$ and $\pi_{\gamma}(A)\bar{\bf b}_{\gamma}\pi_{\gamma}(D)$. As $X_{\gamma}$ is a coarse proximity space, this implies that $\pi_{\gamma}(A)\bar{\bf b}_{\gamma}(\pi_{\gamma}(C)\cup\pi_{\gamma}(D))$. Consequently $\pi_{\gamma}(A)\bar{\bf b}_{\gamma}\pi_{\gamma}(C\cup D)$, which is a contradiction. We then have that ${\bf b}$ satisfies the union axiom and $(Z,\mathcal{B},{\bf b})$ is a coarse proximity space. 
\end{proof}

As is the case with general inverse systems of sets it is very possible for the inverse limit of coarse proximity spaces to be empty, even when dealing with an inverse system indexed by $\mathbb{N}$ and whose connection morphisms are coarse proximity isomorphisms. For example, let $(X,d)$ be any unbounded metric space with fixed point $x_{0}$. For each $n\in\mathbb{N}$ let $X_{n}=X_{n}\setminus B(x_{0},n)$ equipped with the metric coarse proximity structure and let connection morphisms simply be inclusion. Then each connection morphism is a coarse proximity isomorphism, but the set theoretic inverse limit is empty. In the following theorems we will restrict ourselves to the case where the set theoretic inverse limit exists and has surjective projections. This may seem excessively restrictive, but will end up being critical in later sections.

\begin{theorem}\label{inverse limit in coarse proximity category}
Let $\mathfrak{C}$ be the category of coarse proximity spaces and coarse proximity maps. Let $\{(X_{\alpha},\mathcal{B}_{\alpha},{\bf b}_{\alpha}),f_{\beta\alpha}\}_{\alpha,\beta\in \Omega}$ be an inverse system in $\mathcal{C}$ whose set theoretic inverse limit $Z$ exists and is such that the projections $\pi_{\alpha}:Z\rightarrow X_{\alpha}$ are surjective. Then the space $(Z,\mathcal{B},{\bf b})$ as described above is the inverse limit of the inverse system $\{(X_{\alpha}),\mathcal{B}_{\alpha},{\bf b}_{\alpha}),f_{\alpha\beta}\}_{\alpha,\beta\in \Omega}$.
\end{theorem}
\begin{proof}
We begin by noting that by the construction of $Z$ for all $\alpha,\beta\in \Omega$ with $\alpha<\beta$ we have that $f_{\beta\alpha}\circ\pi_{\beta}=\pi_{\alpha}$. Let $(Y,\mathcal{B}_{Y},{\bf b}_{Y})$ be a coarse proximity space, and for each $\alpha\in \Omega$ let $g_{\alpha}:Y\rightarrow X_{\alpha}$ be a coarse proximity map satisfying such that the collection of $g_{\alpha}$ satisfies the same commutativity with the $f_{\alpha\beta}$'s as the $\pi_{\alpha}$'s. Then for a given $y\in Y$ define $g:Y\rightarrow Z$ by $g(y)=(g_{\alpha}(y))_{\alpha\in \Omega}$. It is clear that this map is well defined by the commutativity condition on the $g_{\alpha}$'s. Likewise, it is trivial that $\pi_{\alpha}\circ g=g_{\alpha}$ for every $\alpha$. We must show that $g$ is a coarse proximity map, and is the unique coarse proximity such that $\pi_{\alpha}\circ g_{\alpha}=g$ for each $\alpha$. Uniqueness is clear. We then begin by letting $B\subseteq Y$ be bounded. Then for any given $\alpha$ we have that $g_{\alpha}(B)$ is bounded. As $\pi_{\alpha}\circ g(B)=g_{\alpha}(B)$ we have that the image of $g(B)$ under $\pi_{\alpha}$ is bounded in $X_{\alpha}$. Therefore $g(B)$ is bounded in $Z$. Next let $A,C\subseteq Y$ be unbounded such that $A{\bf b}_{Y}C$. Then $g_{\alpha}(A){\bf b}_{\alpha}g_{\alpha}(C)$ for every $\alpha$. Then $\pi_{\alpha}(g(A)){\bf b}_{\alpha}\pi_{\alpha}(g(C))$ for every $\alpha$, which gives us that $g(A){\bf b}g(C)$. Thus, $g$ is a coarse proximity map. We have then shown that $Z$ is indeed the inverse limit of the system.
\end{proof}

\begin{theorem}\label{inverse limit commutes with boundary functor}
Let $\{(X_{\alpha},\mathcal{B}_{\alpha},{\bf b}_{\alpha}),f_{\beta\alpha}\}_{\alpha,\beta\in \Omega}$ be an inverse system in $\mathcal{C}$ whose set theoretic inverse limit $Z$ exists and is such that the projections $\pi_{\alpha}:Z\rightarrow X_{\alpha}$ are surjective. Equipping $Z$ with the inverse limit coarse proximity structure $(Z,\mathcal{B},{\bf b})$ we have that $\mathcal{U}Z$ is the inverse limit of the inverse system $\{\mathcal{U}X_{\alpha},\mathcal{U}f_{\alpha\beta}\}_{\alpha,\beta\in \Omega}$. 
\end{theorem}
\begin{proof}
That $\{\mathcal{U}X_{\alpha},\mathcal{U}f_{\alpha\beta}\}_{\alpha,\beta\in \Omega}$ is an inverse system is given by the functorality of the boundary. Similarly we have the requisite commutativity of triangles involving $\mathcal{U}\pi_{\alpha}$'s and the $\mathcal{U}f_{\alpha\beta}$'s. Then, letting $Y$ be the inverse limit of the inverse system $\{\mathcal{U}X_{\alpha},\mathcal{U}f_{\alpha\beta}\}_{\alpha,\beta\in \Omega}$ with projections $g_{\alpha}$, we have that there is a unique continuous function $f:\mathcal{U}Z\rightarrow Y$ such that all $g_{\alpha}\circ f=\mathcal{U}\pi_{\alpha}$ for each $\alpha$. We will show that $f$ is a homeomorphism. Let $\sigma_{1},\sigma_{2}\in\mathcal{U}Z$ be given and assume that $f(\sigma_{1})=f(\sigma_{2})$. Then for each $\alpha$ we have that $g_{\alpha}(f(\sigma_{1}))=g_{\alpha}(f(\sigma_{2}))$, which by the commutativity gives us that $\mathcal{U}\pi_{\alpha}(\sigma_{1})=\mathcal{U}\pi_{\alpha}(\sigma_{2})$ for each $\alpha$. If $\sigma_{1}\neq\sigma_{2}$ then there are $A\in\sigma_{1}$ and $C\in\sigma_{2}$ such that $A\bar{\bf b}C$ in $Z$. This implies that there is some $\gamma$ for which $\pi_{\gamma}(A)\bar{\bf b}_{\gamma}\pi_{\gamma}(C)$. We then recall that $\mathcal{U}\pi_{\gamma}(\sigma_{1})=\{D\subseteq X_{\gamma}\mid D{\bf b}_{\gamma}\pi_{\gamma}(E),\,\forall E\in\sigma_{1}\}$. Clearly $\pi_{\gamma}(C)$ will not be in this set, which implies that $\mathcal{U}\pi_{\gamma}(\sigma_{1})$ would not be equal to $\mathcal{U}\pi_{\gamma}(\sigma_{2})$. We then must have that $\sigma_{1}=\sigma_{2}$. We then have that $f$ is injective.
\vspace{\baselineskip}

As $\mathcal{U}Z$ and $Y$ are compact and Hausdorff we have that $f(\mathcal{U}Z)$ is closed in $Y$. Because $g_{\alpha}\circ f=\mathcal{U}\pi_{\alpha}$ for each $\alpha$ and each $\mathcal{U}\pi_{\alpha}$ is surjective we have that $g_{\alpha}(f(\mathcal{U}Z))=\mathcal{U}X_{\alpha}$ for each $\alpha$. Then, by general properties of inverse limits we have that $f(\mathcal{U}Z)=\varprojlim\mathcal{U}X_{\alpha}=Y$. We have then shown that $f$ is a continuous bijection between compact Hausdorff spaces and is therefore a homeomorphism.

\end{proof}

\section{Coarsely $n-to-1$ maps and canonical covers of coarse proximity spaces}\label{coarse n to 1}

This section is dedicated to translating results of Austin and Virk from \cite{austinvirk} to the language of coarse proximity spaces. The results pertain to coarse analogs of $n$-to-$1$ maps between topological spaces. In particular, using the results of Austin and Virk, we can provide conditions under which coarse proximity maps induce $n$-to-$1$ maps between boundaries of the coarse proximity spaces involved. The following three definitions are as they appear in \cite{austinvirk}.

\begin{definition}
Suppose $\mathcal{A}=\{A_{1},\ldots,A_{k}\}$ is a collection of disjoint subsets of a metric space $X$. The collection $\mathcal{A}$ is called {\bf gradually disjoint} if for each $R>0$ there is a bounded set $B\subseteq X$ such that the sets $B(A_{i},R)\setminus B$ are disjoint. 
\end{definition}

\begin{definition}
Suppose $\mathcal{A}=\{A_{1},\ldots,A_{k}\}$ is a collection of subsets of a metric space $X$. We say that the collection $\mathcal{A}$ {\bf diverges} or is {\bf divergent} if there is a bounded set $B\subseteq X$ such that 

\[\bigcap_{i=1}^{k}B(A_{i},R)\setminus B=\emptyset\]
\end{definition}

In a discrete proper metric space $X$, a collection $\mathcal{A}=\{A_{1},\ldots,A_{k}\}$ is gradually disjoint if and only if the traces of the $A_{i}$ on the Higson corona are pairwise disjoint. The collection is divergent if and only if the intersection of the traces of the $A_{i}$ is empty. 
\vspace{\baselineskip}

\begin{definition}
A coarse map $f:X\rightarrow Y$ between discrete proper metric spaces is called {\bf coarsely $n-to-1$} if for every $R>0$ there is an $S>0$ such that if $B\subseteq Y$ satisfies $diam(B)\leq R$ then there are at most $n$ sets $K_{1},\ldots,K_{n}\subseteq X$ such that $diam(K_{i})\leq S$ for each $i$ and $f^{-1}(B)\subseteq\bigcup_{i=1}^{n}K_{i}$. 
\end{definition}

The following theorems appears in \cite{austinvirk}. 

\begin{theorem}\label{metric n to one}
Let $X$ and $Y$ be discrete proper metric spaces. Suppose $f:X\rightarrow Y$ is a coarse map. Then the following are equivalent:
\begin{enumerate}
\item $f$ is coarsely $n-to-1$.
\item $\nu f$ is $n-to-1$.
\item If $\mathcal{A}=\{A_{1},\ldots,A_{k}\}$ is a gradually disjoint collection of subsets of $X$ for which $f(\mathcal{A})=\{f(A_{1}),\ldots,f(A_{k})\}$ is not divergent, then $k\leq n$.
\end{enumerate}
\end{theorem}

\begin{theorem}\label{asdim characterization}
Let $(X,d)$ be an unbounded proper metric space such that there is some $r>0$ for which $d(x,y)>r$ for all $x\neq y$. Then the following are equivalent.
\begin{enumerate}
\item $asdim(X)\leq n$
\item There is a proper metric space $Z$ of asymptotic dimension zero and a coarsely $n+1$-to-$1$ and coarsely surjective map $f:Z\rightarrow X$.
\end{enumerate}

\end{theorem}
 
Theorem \ref{asdim characterization} was actually first proved by Miyata and Virk in \cite{miyatavirk}. A consequence of Theorem \ref{asdim characterization} not stated in \cite{miyatavirk} or \cite{austinvirk} is the following.

\begin{theorem}\label{asdim bigger than cofinal}
If $(X,d)$ is an unbounded proper metric space, then $asdim(X)\geq\Delta(\nu X)$.

\end{theorem}
\begin{proof}
If $asdim(X)=n$, then by Theorem \ref{asdim characterization} there is a metric space $Z$ such that $asdim(Z)=0$ and a coarsely $n+1$-to-$1$ coarsely surjective map $f:Z\rightarrow X$. Then, by Theorem \ref{metric n to one} the continous function $\mathcal{U}f:\mathcal{U}Z\rightarrow\nu X$ is $n+1$-to-$1$. By Proposition \ref{irreducible from perfect} there is a closed subset $A\subseteq Z$ such that $\mathcal{U}f$ restricted to $A$ is irreducible and still $n+1$-to-$1$. By Theorem \ref{cofinal dimension characterization} we then have that $\Delta(\nu X)\leq n$. Therefore $\Delta(\nu X)\leq asdim(X)$. 
\end{proof}

Now we will generalize the above definitions and Theorem \ref{metric n to one} (in part) to coarse proximity spaces.

\begin{definition}
Let $(X,\mathcal{B},{\bf b})$ be a coarse proximity space and $\mathcal{A}=\{A_{1},\ldots,A_{n}\}$ a collection of subsets of $X$. We say that $\mathcal{A}$ is:\\
\begin{enumerate}
\item {\bf gradually disjoint} if for $i\neq j$ we have $A_{i}\bar{\bf b}A_{j}$,
\item {\bf divergent} if $\bigcup_{i=1}^{n}tr(X\setminus A_{i})=\mathcal{U}X$ and there are sets $D_{1},\ldots,D_{n}$ such that $D_{i}\ll (X\setminus A_{i})$ and $\bigcup_{i=1}^{n}tr(D_{i})=\mathcal{U}X$. (i.e. if the set of $X\setminus A_{i}$ is a ${\bf b}$-cover.
\end{enumerate}
\end{definition}

The reader is encourage to compare the definition of divergent families above to the definition of \emph{vanishing families} in the study of uniform spaces as seen in \cite{isbell}.

\begin{proposition}\label{gradually disjoint divergent characterization}
Let $\mathcal{A}=\{A_{1},\ldots,A_{n}\}$ be a family of subsets of a coarse proximity space $(X,\mathcal{B},{\bf b})$ then:\\
\begin{enumerate}
\item $\mathcal{A}$ is gradually disjoint if and only if the traces $tr(A_{i})$ on $\mathcal{U}X$ are disjoint,
\item $\mathcal{A}$ is divergent if and only if $\bigcap_{i=1}^{n}tr(A_{i})=\emptyset$. 
\end{enumerate}
\end{proposition}
\begin{proof}
$(1)$ This equivalence is clear upon recalling that $tr(A_{i})\cap tr(A_{j})\neq\emptyset$ if and only if $A_{i}{\bf b}A_{j}$.
\vspace{\baselineskip}

$(2)$ Assume that $\mathcal{A}=\{A_{1},\ldots,A_{n}\}$ is divergent and assume towards a contradiction that $\sigma\in\bigcap_{i=1}^{n}tr(A_{i})$. Let $D_{1},\ldots,D_{n}$ be such that $D_{i}\ll(X\setminus A_{i})$ for each $i$ and $\bigcup_{i=1}^{n}tr(D_{i})=\mathcal{U}X$. Because the traces of the $D_{i}$ cover $\mathcal{U}X$ we must have that there is some $i$ such that $\sigma\in tr(D_{i})$. Without loss of generality lets say that $\sigma\in tr(D_{1})$. Because $D_{1}\ll (X\setminus A_{1})$ we have that $tr(D_{1})\subseteq int(tr(X\setminus A_{1}))$. Therefore $\sigma\in int(tr(X\setminus A_{1}))$ so there is an open set $U\subseteq\mathcal{U}X$ with $\sigma\in U$ and $cl(U)\subseteq int(tr(X\setminus A_{1}))$. Then $\sigma\notin tr(A_{1})$ which is a contradiction. Therefore $\bigcap_{i=1}^{n}tr(A_{i})=\emptyset$.
\vspace{\baselineskip}

Conversely assume that $\bigcap_{i=1}^{n}tr(A_{i})=\emptyset$ and define $C_{i}=(X\setminus A_{i})$ for each $i$. Let $U_{i}=\mathcal{U}X\setminus tr(A_{i})$ for each $i$. The collection $\mathcal{U}=\{U_{1},\ldots,U_{n}\}$ is an open cover of $\mathcal{U}X$. There is then a finite open cover $\mathcal{V}$ of $\mathcal{U}X$ such that $\{\overline{V}\mid V\in\mathcal{V}\}$ refines $\mathcal{U}$. For each $V\in\mathcal{V}$  with $\overline{V}\subseteq U_{i}$ we can find an unbounded set $A_{V}\subseteq X$ such that $V\subseteq tr(A_{V})\subseteq U_{i}$. Then, for each $V\in\mathcal{V}$ let $\alpha(V)\in\{1,\ldots,n\}$ be such that $\overline{V}\subseteq U_{\alpha(V)}$ and let $A_{V}$ be such that $V\subseteq tr(A_{V})\subseteq U_{\alpha(V)}$. Then define $D_{i}=\bigcup_{\alpha(V)=i}A_{V}$. Then because $\mathcal{V}$ is finite we have that $tr(D_{i})\ll C_{i}$ for each $i$. Similarly $\mathcal{U}X=\bigcup_{i=1}^{n}tr(D_{i})$. Therefore $\{A_{1},\ldots,A_{n}\}$ is divergent. 
\end{proof}

\begin{theorem}\label{finite to one criterion}
Let $f:(X,\mathcal{B}_{X},{\bf b}_{X})\rightarrow (Y,\mathcal{B}_{Y},{\bf b}_{Y})$ be a coarse proximity map. Then the following are equivalent:\\
\begin{enumerate}
\item $\mathcal{U}f$ is $n$-to-$1$
\item If $\mathcal{A}=\{A_{1},\ldots,A_{k}\}$ is a gradually disjoint collection of subsets of $X$ such that $f(\mathcal{A})=\{f(A_{1}),\ldots,f(A_{k})\}$ is not divergent, then $k\leq n$. 
\end{enumerate}
\end{theorem}
\begin{proof}
$(1)\implies (2)$ Assume that $\mathcal{U}f$ is $n$-to-$1$ and let $\mathcal{A}=\{A_{1},\ldots,A_{k}\}$ be a gradually disjoint collection of subsets of $X$ such that $f(\mathcal{A})=\{f(A_{1}),\ldots,f(A_{k})\}$ is not divergent. Assume towards a contradiction that $k>n$. Without loss of generality let us say that $k=n+1$. By Proposition \ref{gradually disjoint divergent characterization} we have that there is some point $\sigma\in\bigcap_{i=1}^{n+1}tr(f(A_{i}))$. Then $\mathcal{U}f^{-1}(\sigma)\cap tr(A_{i})\neq\emptyset$ for each $i$. However, as $\mathcal{A}$ is gradually disjoint and has $n+1$ elements we have that $\mathcal{U}f^{-1}(\sigma)$ has at least $n+1$ elements, contradicting $\mathcal{U}f$ being $n$-to-$1$. Therefore $k\leq n$.
\vspace{\baselineskip}

$(2)\implies(1)$. Assume $(2)$ and assume further towards a contradiction that $\sigma\in\mathcal{U}Y$ is such that $\mathcal{U}f^{-1}(\sigma)$ contains at least $n+1$ elements. As $\mathcal{U}X$ is compact and Hausdorff we can then separate $n+1$ elements of $\mathcal{U}f^{-1}(\sigma)$ with disjoint open neighbourhoods. We can then find a gradually disjoint family of unbounded sets $A_{1},\ldots,A_{n+1}$ in $X$ such that $tr(A_{i})\cap\mathcal{U}f^{-1}(\sigma)\neq\emptyset$ for each $i$. However, then $\{f(A_{1}),\ldots,f(A_{n+1})\}$ is not divergent. This contradicts $(2)$. Therefore $\mathcal{U}f$ must be $n$-to-$1$.
\end{proof}

\section{Coarsely canonical covers and the coarse cofinal dimension}

In this section we will introduce a coarse analog of cofinal dimension for coarse proximity spaces. The definition itself makes use of a coarse analog of almost canonical covers of topological spaces. 

\begin{definition}\label{coarsely canonical cover definition}
A collection $\mathcal{A}=\{A_{1},\ldots,A_{n}\}$ of disjoint subsets of a coarse proximity space $(X,\mathcal{B},{\bf b})$ is a {\bf coarsely canonical cover} of $X$ if:\\
\begin{enumerate}
\item There is a unique, possibly empty, bounded element of $\mathcal{A}$.
\item The union of the $A_{i}$ is equal to $X$.
\item For each unbounded element $A_{i}\in\mathcal{A}$ there is an unbounded $K\subseteq X$ such that $K\ll_{w} A_{i}$.
\item For $i\neq j$ there is no unbounded $K$ such that $K\ll_{w} A_{i}$ and $K\ll_{w}A_{j}$.
\item If $K\subseteq X$ is an unbounded set and $K\ll D$, then there is an unbounded $C\subseteq X$ and some $i$ for which $C\ll D$ and $C\ll_{w} A_{i}$.
\end{enumerate} 

\end{definition}

One may take issue with $(2)$ in the above definition as it doesn't seem particularly ``coarse", however requiring that these collections cover our space will be very useful in the next section. 

\begin{example}
If $(X,d)$ is an unbounded metric space equipped with its metric coarse proximity structure, then $\{X\}$ is a coarsely canonical cover. 
\end{example}
\begin{example}
Let $X=\mathbb{R}^{2}$ be equipped with its metric coarse proximity structure. Define $A_{1}$ to the be left half-plane including the $y$-axis, and $A_{2}$ the right half-plane containing the $y$-axis. Then $\{A_{1},A_{2}\}$ is a coarsely canonical cover of $\mathbb{R}^{2}$. 
\end{example}

\begin{proposition}\label{trace of coarsely canonical cover}
Let $\mathcal{U}=\{A_{1},\ldots,A_{n}\}$ be a coarsely canonical cover of a proper metric space $(X,d)$. Then the collection $tr(\mathcal{U})=\{tr(A_{1}),\ldots,tr(A_{n})\}$ is a almost canonical cover of $\nu X$. Conversely, if $\mathcal{U}=\{A_{1},\ldots,A_{n}\}$ is a cover of $X$ by disjoint sets such that $\{tr(A_{1}),\ldots,tr(A_{n})\}$ is a almost canonical cover of $\nu X$, then $\{A_{1},\ldots, A_{n}\}$ is a coarsely canonical cover of $X$ up to the possible inclusion of a bounded set.
\end{proposition}
\begin{proof}
We begin by proving the result for coarsely canonical covers. Because $\mathcal{U}$ covers $X$ and is finite we have that $tr(\mathcal{U})$ covers $\nu X$. If for $i\neq j$ we have that there is some $\sigma\in int(tr(A_{i}))\cap int(tr(A_{j}))$ there is there an unbounded set $D\subseteq X$ such that $\sigma\in tr(D)\subseteq int(tr(A_{i}))\cap int(tr(A_{j}))$ which implies $D\ll_{w} A_{i}$ and $D\ll_{w}A_{j}$ contradicting item $(2)$ of Definition \ref{coarsely canonical cover definition}. Finally, let $\sigma\in\nu X$ be any element and $U\subseteq\nu X$ some open set containing $\sigma$. If $U$ does not intersect the interior of any $tr(A_{i})$ then we can find unbounded sets $K,D\subseteq X$ such that $K\ll D$ and there is no unbounded $C\subseteq X$ and $i$ for which $C\ll D$ and $C\ll_{w} A_{i}$, contradicting item $(3)$ of Definition \ref{coarsely canonical cover definition}. Therefore $U$ must intersect the interior of some $tr(A_{i})$. We then have that $tr(\mathcal{U})$ is an almost canonical cover of $\nu X$.
\vspace{\baselineskip}

Conversely, assume that $\{A_{1},\ldots,A_{n}\}$ is a cover of $X$ by disjoint sets such that $\{tr(A_{1}),\ldots,tr(A_{n})\}$ is an almost canonical cover of $\nu X$. If $K\subseteq X$ is such that there are $i,j$ for which $K\ll_{w} A_{i}$ and $K\ll_{w}A_{j}$, then $tr(K)\subseteq int(tr(A_{i}))\cap int(tr(A_{j}))$. As the traces of the $A_{k}$ are an almost canonical cover of $\nu X$ they have disjoint interiors, therefore $i=j$. If $K,D\subseteq X$ are such that $K\ll D$, then $tr(K)\subseteq int(tr(D))$. Because the traces of the $A_{i}$ are an almost canonical cover there must be some $\sigma\in int(tr(D))$ and some $i$ such that $\sigma\in int(tr(A_{i}))$. We can then find some unbounded subset $C\subseteq X$ such that $\sigma\in tr(C)\subseteq int(tr(A_{i}))\cap int(tr(D))$ and $C\ll D$ by Lemma \ref{traces to coarse neighbourhoods}. Because $tr(C)\subseteq int(A_{i})$ we also have that $C\ll_{w}A_{i}$. Finally, as the $A_{i}$ cover $X$, we have that $\{A_{1},\ldots,A_{n}\}$ is a coarsely canonical cover of $X$ up to possibly the inclusion of bounded set.

\end{proof}

\begin{definition}\label{order of coarsely canonical cover}
Let $\mathcal{U}=\{A_{1},\ldots,A_{n}\}$ be a coarsely canonical cover of a coarse proximity space $(X,\mathcal{B},{\bf b})$. The {\bf order} of $\mathcal{U}$, denoted $ord(\mathcal{U})$, is the least natural number $m$ that if $A_{i_{1}},\ldots,A_{i_{m+1}}$ is any collection of $m+1$ elements of $\mathcal{U}$, the the $A_{i_{j}}$ diverge.

\end{definition}

\begin{proposition}\label{order relation}
If $\mathcal{U}=\{A_{1},\ldots,A_{n}\}$ is a coarsely canonical cover of a metric space $X$, then $ord(\mathcal{U})=m$ if and only if the order of the collection $tr(\mathcal{U})=\{tr(A_{1}),\ldots,tr(A_{n})\}$ is $m$. 
\end{proposition}

\begin{proof}
Assume that $ord(\mathcal{U})\leq m$ and let $tr(A_{i_{1}}),\ldots,tr(A_{i_{m+1}})$ be $m+1$ elements of $tr(\mathcal{U})$. If $\bigcap_{j=1}^{m+1}tr(A_{i_{j}})\neq\emptyset$, then by Proposition \ref{gradually disjoint divergent characterization} we have that $\{A_{i_{1}},\ldots,A_{i_{m+1}}\}$ is not a divergent family which is a contradiction. Therefore every collection of $m+1$ elements of $tr(\mathcal{U})$ must have nonempty intersection, so $ord(tr(\mathcal{U}))\leq m$. Conversely assume that $ord(tr(\mathcal{U})\leq m$ and let $A_{i_{1}},\ldots,A_{i_{m+1}}$ be $m+1$ elements of $\mathcal{U}$. If these sets are not divergent, then by Proposition \ref{gradually disjoint divergent characterization} $\bigcap_{j=1}^{m+1}tr(A_{i_{j}})$ is nonempty, which contradicts $ord(tr(\mathcal{U}))\leq m$. 
\end{proof}

\begin{definition}\label{coarse refinement}
Given two collections of subsets $\mathcal{U}=\{A_{1},\ldots,A_{n}\}$ and $\mathcal{V}=\{C_{1},\ldots,C_{m}\}$ of a coarse proximity space $(X,\mathcal{B},{\bf b})$ we say that $\mathcal{V}$ {\bf coarsely refines} $\mathcal{U}$ if for each unbounded $C_{i}\in\mathcal{V}$ there is an unbounded $A_{j}\in\mathcal{U}$ and a bounded set $B\subseteq X$ such that $(C_{i}\setminus B)\subseteq A_{j}$. 

\end{definition}

It is possible that two distinct coarsely canonical covers may coarsely refine each other. Fortunately, with respect to coarse refinement, this equivalence does not matter as can be seen in the following easy proposition.

\begin{proposition}\label{equivalent covers are very equivalent}
Let $\mathcal{U}$ and $\mathcal{V}$ be coarsely canonical covers of a coarse proximity space $(X,\mathcal{B},{\bf b})$ that coarsely refine each other. For a third coarsely canonical cover $\mathcal{W}$ the following are equivalent:\\
\begin{enumerate}
\item $\mathcal{W}$ coarsely refines $\mathcal{U}$
\item $\mathcal{W}$ coarsely refines $\mathcal{V}$

\end{enumerate}

\end{proposition}

\begin{proposition}\label{refining open cover with canonical}
Let $\mathcal{U}=\{A_{1},\ldots,A_{n}\}$ be a ${\bf b}$-cover of a coarse proximity space $(X,\mathcal{B},{\bf b})$ that covers $X$. Then there is a coarsely canonical cover $\mathcal{V}=\{C_{1},\ldots,C_{m}\}$ that coarsely refines $\mathcal{U}$.
\end{proposition}
\begin{proof} By Proposition \ref{b cover to open cover} we have that $int(tr(\mathcal{U}))=\{int(tr(A_{1})),\ldots,int(tr(A_{n}))\}$ is an open cover $\mathcal{U}X$. Reduce this open cover to an open cover with the property that every element of the cover contains an element of $\mathcal{U}X$ not contained in any of the other elements of the cover. Call this cover $\mathcal{W}$. Say $\mathcal{W}=\{U_{1},\ldots,U_{m}\}$. Because $\mathcal{U}X$ is a compact Hausdorff space it is normal, so $\mathcal{W}$ admits a shrinking $\mathcal{V}=\{V_{1},\ldots,V_{m}\}$. That is, $\mathcal{V}$ is an open cover of $\mathcal{U}X$ such that $\overline{V}_{i}\subseteq U_{i}$. As each $U_{i}$ contained an element not contained in any other $U_{j}$ we have that every $V_{j}$ contains a nontrivial open set that does not intersect any other $V_{j}$. Note that $\mathcal{V}$ trivially refines $int(tr(\mathcal{U}))$. By Lemma \ref{closed covers admit canonical refinements} we have that there is a canonical cover $\mathcal{F}=\{F_{1},\ldots,F_{k}\}$ of $\mathcal{U}X$ that refines $\{\overline{V}_{1},\ldots,\overline{V}_{m}\}$, and consequently $\mathcal{W}$ and $int(tr(\mathcal{U}))$. 
\vspace{\baselineskip}

By Proposition \ref{preliminary proposition} can find an unbounded set $D_{1}\subseteq X$ such that $F_{1}\subseteq int(tr(D_{1}))$, $tr(D_{1})\cap F_{i}=\emptyset$ for all $F_{i}$ such that $F_{i}\cap F_{1}=\emptyset$, is such that $tr(D_{1})$ is contained in some $U_{i}\in\mathcal{W}$, and is such that if $K\subseteq X$ is an unbounded set with $tr(K)\subseteq F_{1}$, then $K\ll D_{1}$. Assume that unbounded sets $D_{1},\ldots, D_{k}$ have been constructed such that for $i\leq k$:
\begin{enumerate}
\item $F_{i}\subseteq int(tr(D_{i}))$
\item $D_{i}{\bf b}D_{j}$ if and only if $F_{i}\cap F_{j}\neq\emptyset$.
\item $tr(D_{i})\cap F_{j}\neq\emptyset$ implies $F_{i}\cap F_{j}\neq\emptyset$
\item $tr(D_{i})$ is contained in some $U_{j}\in\mathcal{W}$.
\item If $K\subseteq X$ is such that $tr(K)\subseteq F_{i}$, then $K\ll D_{i}$
\end{enumerate}

Consider $F_{k+1}$. As $F_{k+1}$ is disjoint from all $D_{i}$ for which $F_{k+1}\cap F_{i}=\emptyset$ we can find an unbounded set $D_{k+1}$ such that $F_{k+1}\subseteq int(tr(D_{k+1}))$, $D_{k+1}{\bf b} D_{i}$ if and only if $F_{k+1}\cap D_{i}\neq\emptyset$, and $D_{k+1}\cap F_{j}\neq\emptyset$ implies $F_{k+1}\cap F_{j}\neq\emptyset$. By induction we then have a collection of unbounded sets $\mathcal{D}=\{D_{1},\ldots, D_{m_{2}}\}$ satisfying items $(1)$, $(2)$, and $(4)$ as above. Without loss of generality we may assume that for every $i$ there is some open set $V\subseteq tr(D_{i})$ that does not intersect any other $tr(D_{j})$. We may also assume that the collection of $D_{i}$'s covers $X$. Let $S=\{1,\ldots,m_{2}\}$. Because $\mathcal{F}$ is finite it has finite order, say $k+1$. We then consider all nonempty unbounded intersections of the form $\bigcap_{i\in F\subseteq S,\,|F|=k+1}D_{i}$. Let $\mathcal{C}_{k+1}$ be the collection of these unbounded sets. We define $\hat{\mathcal{C}}_{k}$ be the similarly defined collection of unbounded sets and define $\mathcal{C}_{k}$ to be the collection of sets of the forum $C\setminus\left(\bigcup_{F\in\mathcal{C}_{k+1}}F\right)$ where $C\in\hat{\mathcal{C}}_{k}$. We continue in like fashion constructing collections of unbounded sets $\mathcal{C}_{i}$ for all $i\in S$ with the property that every every of $\mathcal{C}_{i}$ is disjoint from every element of $\bigcup_{j>i}\mathcal{C}_{j}$. We finally let $\hat{\mathcal{F}}=\bigcup_{i=1}^{k+1}\mathcal{C}_{i}$. This is a finite set of unbounded sets whose corresponding set of traces refines $\mathcal{U}$. We claim that $\hat{\mathcal{F}}$ is the desired coarsely canonical cover. It is clear from construction that for distinct $C_{1},C_{2}\in\hat{\mathcal{F}}$ there is no unbounded set $K\subseteq X$ such that $K\ll_{w}C_{1}$ and $K\ll_{w}C_{2}$. It is also clear that $\hat{\mathcal{F}}$ covers $X$. Let $K,D\subseteq X$ be unbounded sets with $K\ll D$. Consider $int(tr(D))$. As $\{int(tr(D_{1})),\ldots,int(tr(D_{m_{2}}))\}$ is an open cover of $\nu X$ there is a largest natural number of $k_{\sigma}$ such that there is a point $\sigma\in\nu X$ for which $\sigma$ is the in the intersection of $int(tr(D))$ and exactly $k_{\sigma}-1$ open sets of the form $int(tr(D_{i}))$. There is an element $C$ of $\hat{\mathcal{F}}$ that contains this intersection. We may then find an unbounded set $E\subseteq X$ such that $E\ll D$ and $E\ll_{w}C$. Therefore $\hat{\mathcal{F}}$ is a coarsely canonical cover. What remains to be shown is that $\hat{\mathcal{F}}$ coarsely refines $\mathcal{U}$. However, this is trivial as $tr(D_{i})\subseteq U_{j}=int(tr(A_{j}))$ for some $F_{j}\in\mathcal{U}$ which implies $D_{i}\ll A_{j}$ which implies that $D_{i}$ is contained in $A_{j}$ up to a bounded set.

\end{proof}

\begin{definition}\label{fine system of covers}
We say that a pre-directed system of coarsely canonical covers $\{F_{\alpha}\}_{\alpha\in\Omega}$ is {\bf fine} if for every ${\bf b}$-cover $\mathcal{U}$ that covers $X$ there is some $\alpha\in\Omega$ such that $\mathcal{F}_{\alpha}$ coarsely refines $\mathcal{U}$.
\end{definition}

\begin{definition}\label{coarse cofinal definition}
Let $(X,\mathcal{B},{\bf b})$ be a coarse proximity space. We define the {\bf coarse cofinal dimension} of $X$, denoted $\Delta_{c}(X)$, in the following way. If $X$ is bounded then $\Delta_{c}(X)=-1$. Otherwise $\Delta_{c}(X)$ is the smallest nonnegative integer $n$ such that $X$ admits a fine directed system of coarsely canonical covers $\{\mathcal{F}_{\alpha}\}_{\alpha\in\Omega}$ such that the order of $\mathcal{F}_{\alpha}$ is at most $n+1$ for each $\alpha\in\Omega$. If there is no such $n$ we say that $\Delta_{c}(X)=\infty$. 
\end{definition}

 Definition \ref{coarse cofinal definition}, Proposition \ref{trace of coarsely canonical cover}, and Proposition \ref{order relation} together with the observation that taking traces of coarsely canonical covers commutes with the refinement relation yield the following.

\begin{theorem}\label{coarse cofinal larger than cofinal of corona}
If $X$ is a proper metric space equipped with its metric coarse proximity structure, then $\Delta_{c}(X)\geq \Delta(\nu X)$ where $\nu X$ is the Higson corona of $X$.
\end{theorem}

Next we will work towards proving a basic equality. Namely, that for unbounded discrete proper metric spaces, asymptotic dimension and coarse cofinal dimension agree in dimension $0$. We will make use of the fact that when a proper metric space has asymptotic dimension $0$, its Higson corona has covering dimension $0$, as given to us by Theorem \ref{asdim and covering dimension}.

\begin{proposition}\label{traceable disjoint covers}
Let $(X,d)$ be an unbounded $1$-discrete proper metric space and let $\mathcal{U}=\{U_{1},\ldots,U_{m}\}$ be a disjoint open cover of $\nu X$. Then there is a coarsely canonical cover $\{A_{1},\ldots,A_{m},A_{m+1}\}$ of $X$, where $tr(A_{i})=U_{i}$ for $i\leq m$ and $A_{m+1}$ is a bounded set.
\end{proposition}
\begin{proof}
Because the elements of the finite open cover $\mathcal{U}$ are disjoint, they are all also closed. By Proposition \ref{preliminary proposition} there is an unbounded set $A_{1}\subseteq X$ such that $tr(A_{1})$ contains $U_{1}$ and is disjoint from $U_{j}$ for $j>1$. Then $tr(A_{1})=U_{1}$. Assume that $A_{i}$ has been constructed for $i\leq n<m$ in such way that the $A_{i}$ are disjoint and are such that if $K\subseteq X$ satisfies $tr(K)\subseteq U_{i}$, then $K\ll A_{i}$. We can again, by Proposition \ref{preliminary proposition} find an unbounded set $\hat{A}_{n+1}\subseteq X$ such that $tr(\hat{A}_{n+1})=U_{n+1}$. As the $U_{i}$ are all disjoint we have that $\left(\bigcup_{i\leq n}A_{i}\right)\bar{\bf b}\hat{A}_{n+1}$ where ${\bf b}$ is the metric coarse proximity on $X$. Then, $\bigcup_{i\leq n}A_{i}$ intersects $\hat{A}_{n+1}$ in at most a a bounded set. Subtract this bounded set from $\hat{A}_{n+1}$ to get $A_{n+1}$. Proposition \ref{preliminary proposition} gives us that if $K\subseteq X$ is such that $tr(K)\subseteq U_{n+1}$, then $K\ll A_{n+1}$. By induction we then have our unbounded sets $A_{1},\ldots,A_{m}$. By construction we have that $X\setminus\left(\bigcup_{i=1}^{m}A_{i}\right)$ is bounded. We then set $A_{m+1}$ to be this bounded set. This completes the proof.
\end{proof}

\begin{proposition}\label{traceable intersection of covers dimension zero}
Let $(X,d)$ be an unbounded $1$-discrete proper metric space of asymptotic dimension $0$, and let $\mathcal{U}=\{A_{1},\ldots,A_{m_{1}}\}$ and $\mathcal{V}=\{C_{1},\ldots,C_{m_{2}}\}$ be coarsely canonical covers of $X$ of order $1$. Then there is a coarsely cofinal cover $\mathcal{W}$ of $X$ that coarsely refines both $\mathcal{U}$ and $\mathcal{V}$.
\end{proposition}
\begin{proof}
We begin by noting that by Proposition \ref{order relation} we have that $tr(\mathcal{U})=\{tr(A_{1}),\ldots,tr(A_{m_{1}})\}$ and $tr(\mathcal{V})=\{tr(C_{1}),\ldots,tr(C_{m_{2}})\}$ are almost canonical covers of $\nu X$ of order $1$ and are therefore disjoint open covers. Then, the set of nonempty intersections $\mathcal{C}=\{tr(A_{i})\cap tr(C_{j})\mid 1\leq i\leq m_{1},\,1\leq j\leq m_{2}\}$ is another disjoint open cover of $\nu X$. By Proposition \ref{traceable disjoint covers} there is a coarsely canonical cover $\hat{\mathcal{W}}=\{\hat{D}_{1},\ldots,\hat{D}_{m_{3}}\}$ whose corresponding set of traces is equal to $\mathcal{C}$. If $tr(D_{k})=tr(A_{i})\cap tr(C_{j})$ then $D_{k}\ll A_{i}$ and $D_{k}\ll C_{j}$. Therefore, $D_{k}\ll (A_{i}\cap C_{j})$ which implies that $D_{k}$ is contained in $A_{i}\cap C_{j}$ up to a bounded set. Define $D_{k}=\hat{D}_{k}\cap(A_{i}\cap C_{j})$ and define $\mathcal{W}$ to be the collection of $D_{k}$'s with $D_{m_{3}}$ being equal to $X\setminus\left(\bigcup_{i<m_{3}}D_{i}\right)$.
\end{proof}

\begin{theorem}\label{asdim and coarse cofinal agree in dimension zero}
Let $(X,d)$ be an unbounded $1$-discrete proper metric space of asymptotic dimension $0$. Then $\Delta_{c}(X)=0$. 
\end{theorem}
\begin{proof}
Let $X$ be given. Then by Theorem \ref{asdim and covering dimension} we have that $dim(\nu X)=0$. Denote by $Cov(\nu X)$ the collection of finite disjoint open covers of $\nu X$, ordered by refinement. Denote by $CCCov(X,1)$ the collection of coarsely canonical covers of $X$ or order $1$, ordered by coarse refinement. As the order defined by coarse refinement is not antisymmetric we let $\mathcal{L}\subseteq CCCov(X,1)$ be such that no two elements of $\mathcal{L}$ coarsely refine each other and every element of $CCCov(X,1)$ is coarsely refined some element of $\mathcal{L}$. Proposition \ref{equivalent covers are very equivalent} assures us that this is possible. By Proposition \ref{traceable disjoint covers} we have that for every $\mathcal{U}\in Cov(\nu X)$ there is a $\mathcal{V}\in \mathcal{L}$ such that the collection of traces of elements of $\mathcal{V}$ is exactly $\mathcal{U}$. If $\mathcal{U},\mathcal{V}\in \mathcal{L}$ are given coarsely canonical covers of $X$, then by Proposition \ref{traceable intersection of covers dimension zero} there is a coarsely canonical cover $\mathcal{W}$ of $X$ coarsely refines both $\mathcal{U}$ and $\mathcal{V}$. Moreover, this coarsely canonical cover has order $1$. We then have that $\mathcal{L}$ is a fine directed system of coarsely canonical covers of $X$ none of which has order greater than $1$. Therefore $\Delta_{c}(X)\leq 0$ and as $X$ is unbounded we have that $\Delta_{c}(X)=0$.

\end{proof}

\section{From coarsely canonical covers to an asymptotic dimension zero coarse proximity space}\label{equality of coarse cofinal and asdim}

This section is focused on proving our main result, that $\Delta_{c}(X)\geq asdim(X)$ for every unbounded proper metric space $X$. We do this by providing a coarse version of a result in classic dimension theory. Specifically we ``coarsify" the result that says that if a space $X$ satisfies $\Delta(X)\leq n$, then there is a perfectly zero dimensional space $Y$ and a closed $n+1$-to-$1$ map $f:Y\rightarrow X$. Our coarse version of the result is made by following the methods for the classical result as thoroughly explained in Chapter $6$ of \cite{pears}. What we show in this section specifically is the following, if $X$ is an unbounded proper metric space satisfying $\Delta_{c}(X)\leq n$, then there is another unbounded proper metric space $Z$ of asymptotic dimension zero (and satisfying $\Delta_{c}(Z)=0$) and a coarsely surjective and coarsely $n+1$-to-$1$ map $f:Z\rightarrow X$.  This will give us that $asdim(X)\leq n$ by Theorem \ref{asdim characterization}. This in conjunction with Theorem \ref{asdim bigger than cofinal} will yield that for every unbounded proper metric space $X$, the inequality $\Delta(\nu X)\leq asdim(X)\leq\Delta_{c}(X)$ holds. 
\vspace{\baselineskip}

Let $\{\mathcal{F}_{\alpha}\}_{\alpha\in\Omega}$ be a fine directed system of coarsely canonical covers of a countable $1$-discrete proper metric space $X$ where $\beta\geq\alpha$ implies that $\mathcal{F}_{\alpha}$ is coarsely refined by $\mathcal{F}_{\beta}$. Say $\mathcal{F}_{\alpha}=\{A_{1},\ldots,A_{n}\}$ and $\mathcal{F}_{\beta}=\{C_{1},\ldots,C_{m}\}$. For a given $\alpha\in\Omega$ we will denote by $\Omega(\alpha)$ the index set of the coarsely canonical cover $\mathcal{F}_{\alpha}$. We will assume that every $\Omega(\alpha)$ is a set of natural numbers of the form $\{1,\ldots,n\}$ for some natural number $n$.  Condition $(1)$ of Definition \ref{coarsely canonical cover definition} makes it so that if $\mathcal{F}_{\alpha}$ refines $\mathcal{F}_{\beta}$ then for each $A_{i}\in\mathcal{F}_{\alpha}$ there is a unique $C_{j}\in\mathcal{F}_{\beta}$ such that $A_{i}\subseteq C_{j}$. There is then a unique function from $\Omega(\beta)=\{1,\ldots,m\}$ to $\Omega(\alpha)=\{1,\ldots, n\}$ denoted by $\rho_{\alpha\beta}$ such that $C_{i}\subseteq A_{\rho_{\alpha\beta}(i)}$ for each $i$. For a given $\alpha$ we then define a coarse proximity space $(X_{\alpha},\mathcal{B}_{\alpha},{\bf b}_{\alpha})$ by define $X_{\alpha}=\bigcup_{i=1}^{n}(A_{i}\times\{i\})$. That is, $X_{\alpha}$ is the disjoint union of the elements of $\mathcal{F}_{\alpha}$. Recall that coarsely canonical covers of $X$ are made up of disjoint sets that cover $X$. The function $\delta_{\alpha}:X_{\alpha}\rightarrow X$ defined by $(x,i)\mapsto x$ is therefore bijective . We say that a subset $B\subseteq X_{\alpha}$ if $\delta_{\alpha}(B)$ is bounded in $X$. It is easily seen that this is a bornology on $X_{\alpha}$. For unbounded sets $A,C\subseteq X_{\alpha}$ we say that $A{\bf b}_{\alpha}C$ if and only if $(\delta_{\alpha}(A)\cap A_{i}){\bf b}(\delta_{\alpha}(C)\cap A_{i})$ for some $i$. It is again easily seen that this relation makes $(X_{\alpha},\mathcal{B}_{\alpha},{\bf b}_{\alpha})$ into a coarse proximity space. With this coarse proximity structure $\delta_{\alpha}$ is a bijective coarse proximity map.
\vspace{\baselineskip}

\begin{lemma}\label{uniformly bounded cover lemma}
A collection $\mathcal{W}$ of bounded subsets of $X_{\alpha}$ is uniformly bounded if and only if the collection 

\[\{\delta_{\alpha}(B)\cap A_{i}\mid B\in\mathcal{W},\,1\leq i\leq n\}\]

is uniformly bounded in $X$ and the collection 

\[\{B\in\mathcal{W}\mid \exists i\neq j,\,B\cap(A_{i}\times{i})\neq\emptyset,\,B\cap(A_{j}\times\{j\})\neq\emptyset\}\]

is finite.
\end{lemma}
\begin{proof}
Assume $\mathcal{W}$ is a collection of subsets of $X_{\alpha}$ that is uniformly bounded. Because $\delta_{\alpha}$ is a coarse proximity map we have that $\partial_{\alpha}(\mathcal{U})$ is uniformly bounded. As $\{\delta_{\alpha}(B)\cap A^{\alpha}_{i}\mid B\in\mathcal{W},\,1\leq i\leq n\}$ refines $\delta_{\alpha}(\mathcal{W})$ it too must be uniformly bounded. If $\{B\in\mathcal{W}\mid \exists i\neq j,\,B\cap(A_{i}\times{i})\neq\emptyset,\,B\cap(A_{j}\cap\{j\})\neq\emptyset\}$  is infinite then by the pigeon hole principle there are $i\neq j$ and an infinite collection of subsets $\hat{\mathcal{W}}\subseteq\mathcal{W}$ such that $B$ contains elements of both $A^{\alpha}_{i}\times\{i\}$ and $A^{\alpha}_{j}\times\{j\}$ for all $B\in\hat{\mathcal{W}}$. We can then find an unbounded subset $K_{1}\subseteq A^{\alpha}_{i}\times\{i\}$ and $K_{2}\subseteq A^{\alpha}_{j}\times\{j\}$ such that $K_{1}\subseteq st(K_{2},\hat{\mathcal{W}})$ and $K_{2}\subseteq st(K_{1},\hat{\mathcal{W}})$. This implies that $K_{1}{\bf b}_{\alpha}K_{2}$. However, this is impossible as $K_{1}\overline{\bf b}_{\alpha}K_{2}$ by the definition of ${\bf b}_{\alpha}$. 
\vspace{\baselineskip}

Conversely, suppose $\mathcal{W}$ is a collection of sets satisfying the two conditions stated in the Lemma. Because $\delta_{\alpha}$ is a coarse proximity map we have that each element of $\mathcal{W}$ is bounded. Let $K_{1},K_{2}\subseteq X_{\alpha}$ be unbounded sets such that $K_{1}\subseteq st(K_{2},\mathcal{W})$ and $K_{2}\subseteq st(K_{1},\mathcal{W})$. We wish to show that $K_{1}\phi_{\alpha}K_{2}$ where $\phi_{\alpha}$ is the weak asymptotic resemblance induced by ${\bf b}_{\alpha}$. Assume towards a contradiction that this is not true. Then, without loss of generality, there is an unbounded $\hat{K}_{1}\subseteq K_{1}$ such that $\hat{K}_{1}\bar{\bf b}_{\alpha}K_{2}$. We may also assume that $\hat{K}_{1}$ does not intersect any of the elements of $\mathcal{W}$ that intersect more than one $A_{i}\times\{i\}$. Let $\tilde{\mathcal{W}}$ be the collection of elements of $\mathcal{W}$ that do not intersect more than one $A_{i}\times\{i\}$ we have that $\hat{K}_{1}\subseteq st(K_{2},\tilde{\mathcal{W}})$. Using the pigeonhole principle we may assume that $\hat{K}_{2}$ is entirely contained in one $A_{i}\times\{i\}$. Again, without loss of generality, this is $A_{1}\times\{1\}$. Then there is an unbounded set $\hat{K}_{2}$ contained entirely in $A_{1}\times\{1\}$ such that $\hat{K}_{1}\subseteq st(\hat{K}_{2},\tilde{\mathcal{W}})$. However, this implies that $(\delta_{\alpha}(\hat{K}_{1})\cap A_{1}){\bf b}(\delta_{\alpha}(\hat{K}_{2})\cap A_{1})$ which implies $\hat{K}_{1}{\bf b}_{\alpha}\hat{K}_{2}$ and consequently that $\hat{K}_{1}{\bf b}_{\alpha}K_{2}$, a contradiction. Therefore $K_{1}\phi_{\alpha}K_{2}$ and $\mathcal{W}$ is uniformly bounded in $X_{\alpha}$.
\end{proof}

Returning to our directed set $\{F_{\alpha}\}_{\alpha\in\Omega}$ of coarsely canonical covers, let $\beta\geq\alpha$ and consider the corresponding coarse proximity spaces $X_{\beta}$ and $X_{\alpha}$. The function $\rho_{\alpha\beta}$ induces a function $\pi_{\alpha\beta}:X_{\beta}\rightarrow X_{\alpha}$ in the following way. If $(x,i)\in X_{\beta}$ then $\pi_{\alpha\beta}(x,i)=(x,\rho_{\alpha\beta}(i))$ is well defined and bijective. Moreover, it is easily seen that $\pi_{\alpha\beta}$ is a coarse proximity map from $X_{\beta}$ to $X_{\alpha}$ such that the collection $\{(X_{\alpha},\mathcal{B}_{\alpha},{\bf b}_{\alpha}),\pi_{\alpha\beta}\}_{\alpha\beta\in\Omega}$ is an inverse system of coarse proximity spaces. Let $(X_{\infty},\mathcal{B}_{\infty},{\bf b}_{\infty})$ be the corresponding inverse limit coarse proximity space. Because each $\rho_{\beta\alpha}$ is bijective and $(\delta_{\alpha}\circ \pi_{\beta\alpha})=\delta_{\beta}$ for all $\beta\geq\alpha$ we have that elements of $X_{\infty}$ can be expressed uniquely as $(x,i,\alpha)_{\alpha\in\Omega}$ where $x\in X$, and $i\in\Omega(\alpha)$ is such that $x\in A^{\alpha}_{i}$. We can then define a map $\delta:X_{\infty}\rightarrow X$ by $(x,i,\alpha)_{\alpha\in\Omega}\mapsto x$.  It is easily seen that $\delta$ is a bijective coarse proximity map. 
\vspace{\baselineskip}

Note that for each $\alpha\in\Omega$ we have that $\mathcal{U}X_{\alpha}$ is homeomorphic to the disjoint union of the boundaries of the elements of $\mathcal{F}_{\alpha}$ by Proposition \ref{boundary of disjoint union is disjoint union}. By Theorem \ref{inverse limit commutes with boundary functor} we have that $\mathcal{U}X_{\infty}$ is homeomorphic to the inverse limit of these disjoint unions.

\begin{lemma}\label{weak ASR at limit}
Given $A,C\subseteq X_{\infty}$ we have that $A\phi_{\infty}C$ if and only if $\pi_{\alpha}(A)\phi_{\alpha}\pi_{\alpha}(C)$ for all $\alpha$.
\end{lemma}
\begin{proof}
Weak asymptotic resemblances are preserved by coarse proximity maps, so one direction is clear. Assume that $\pi_{\alpha}(C)\phi_{\alpha}\pi_{\alpha}(A)$ for all $\alpha$. If $\hat{C}\subseteq C$ is an unbounded subset of $X_{\infty}$ then $\pi_{\alpha}(\hat{C}){\bf b}_{\alpha}\pi_{\alpha}(A)$ for all $\alpha$ because $\pi_{\alpha}(C)\phi_{\alpha}\pi_{\alpha}(A)$ and $\pi_{\alpha}(\hat{C})\subseteq\pi_{\alpha}(C)$. Then $\hat{C}{\bf b}_{\infty}A$. Similarly one shows that if $\hat{A}\subseteq A$ is unbounded, then $C{\bf b}_{\infty}\hat{A}$, so $C\phi_{\infty}A$.
\end{proof}

\begin{lemma}\label{unif bounded covers of limit}
A collection $\mathcal{U}$ of bounded subsets of $X_{\infty}$ is uniformly bounded if and only if $\pi_{\alpha}(\mathcal{U})$ is a uniformly bounded collection of subsets of $X_{\alpha}$ for every $\alpha\in\Omega$.
\end{lemma}
\begin{proof}
Uniformly bounded collections are preserved by coarse proximity maps, so one direction is clear. Assume that $\pi_{\alpha}(\mathcal{U})$ is uniformly bounded in $X_{\alpha}$ for each $\alpha$. Then, for every $B\in\mathcal{U}$ we must have that $\pi_{\alpha}(B)$ is bounded in $X_{\alpha}$, so $B$ is bounded in $X_{\infty}$. Next, if $A,C\subseteq X_{\infty}$ are such that $A\subseteq st(C,\mathcal{U})$ and $C\subseteq st(A,\mathcal{U})$ then $\pi_{\alpha}(A)\subseteq st(\pi_{\alpha}(C),\pi_{\alpha}(\mathcal{U}))$ and $\pi_{\alpha}(C)\subseteq st(\pi_{\alpha}(A),\pi_{\alpha}(\mathcal{U}))$. As $\pi_{\alpha}(\mathcal{U})$ is uniformly bounded for each $\alpha$ we have that $\pi_{\alpha}(A)\phi_{\alpha}\pi_{\alpha}(C)$ for each $\alpha$, which by Lemma \ref{weak ASR at limit} implies that $A\phi_{\infty}C$. Therefore, $\mathcal{U}$ is uniformly bounded in $X_{\infty}$.
\end{proof}

\begin{proposition}
Let $\mathcal{V}$ be a uniformly bounded cover of $X_{\infty}$, then the collection of $\mathcal{V}$-components of $X_{\infty}$, denoted $\mathcal{C}(\mathcal{V})$, is uniformly bounded.
\end{proposition}
\begin{proof}
Let $\mathcal{V}$ be a uniformly bounded cover of $X_{\infty}$ and assume towards a contradiction that $\mathcal{C}(\mathcal{V})$ is not uniformly bounded. Then one of two things must be true. Either $\mathcal{C}(\mathcal{V})$ contains an unbounded set, or there are unbounded sets $A,C\subseteq X_{\infty}$ with $A\subseteq st(C,\mathcal{C}(\mathcal{V}))$ and $C\subseteq st(A,\mathcal{C}(\mathcal{V}))$, but $A\overline{\phi_{\infty}} C$. We will consider these cases separately. Note that for a uniformly bounded cover $\mathcal{U}$ of a space $Y$, by a $\mathcal{U}$-path of length $m$ in $Y$ we mean a sequence of points $y_{0},y_{1},\ldots,y_{m}\in Y$ such that $y_{i-1},y_{i+1}\in st(y_{i},\mathcal{U})$ for $1\leq i\leq m-1$ and is such that $y_{m}\notin st^{m-1}(y_{0},\mathcal{U})$. 
\vspace{\baselineskip}

{\bf Case 1: ($\mathcal{C}(\mathcal{V})$ contains an unbounded set)} In this case there is an unbounded $\mathcal{V}$-component in $X$. Let us denote this component by $K$. Then $K$ is necessarily infinite and $f(K)$ is an unbounded $f(\mathcal{V})$-component in $X$. We can then find an $x_{0}\in f(K)$ and an unbounded $A\subseteq f(K)$ that contains $x_{0}$ and is such that for every $m\in\mathbb{N}$ we can find a $\mathcal{V}$-path of length $m$ in $K$ outside $B(x_{0},n)$. There are then disjoint  unbounded sets $f(C_{1}),f(C_{2})\subseteq A$ such that $f(C_{1})\overline{\bf b}f(C_{2})$ and for every $m,n\in\mathbb{N}$ there is an $f(\mathcal{V})$-path of length $m$ outside of $B(x_{0},n)$ between $f(C_{1})$ and $f(C_{2})$. Then $C_{1}\overline{\bf b}_{\infty}C_{2}$ and for every $m,n\in\mathbb{N}$ there is a $\mathcal{V}$-path of length $m$ outside of $f^{-1}(B(x_{0},n))$ of length $m$. As $f(C_{1})\bar{\bf b}f(C_{2})$ there is a ${\bf b}$-cover $\{D_{1},D_{2},D_{3}\}$ of $X$ such that $f(C_{1})\ll D_{1}$, $f(C_{2})\ll D_{2}$, $f(C_{1})\bar{\bf b}D_{2}$, $f(C_{2})\bar{\bf b} D_{1}$, and $(f(C_{1})\cup f(C_{2}))\bar{\bf b}D_{3}$. Such a cover can be constructed by using the strong axiom for the coarse proximity on $X$. As $\{\mathcal{F}_{\alpha}\}_{\alpha\in\Omega}$ is a fine collection of coarse canonical covers, there is an $\alpha\in\Omega$ such that $\mathcal{F}_{\alpha}=\{A_{1},\ldots,A_{n}\}$ refines this ${\bf b}$-cover. Then $\pi_{\alpha}(C_{1})$ and $\pi_{\alpha}(C_{2})$ are contained in different $A_{i}\in\mathcal{F}_{\alpha}$. As there are $\mathcal{V}$-paths of length $m$ between $C_{1}$ and $C_{2}$ outside of $f^{-1}(B(x_{0},k)$ for every $k$ and $m$ in $\mathbb{N}$ and $f^{-1}(B(x_{0},k))$ exhausts the bounded sets of $X_{\infty}$ we have that a similar condition holds for $\pi_{\alpha}(C_{1})$ and $\pi_{\alpha}(C_{2})$ in $X_{\alpha}$. However, this would imply that the number of elements of $\pi_{\alpha}(\mathcal{V})$ that intersect multiple $A_{i}\in\mathcal{F}_{\alpha}$. However, $\mathcal{V}$ is a uniformly bounded cover of $X_{\infty}$, so this contradicts Lemma \ref{uniformly bounded cover lemma}. Therefore $\mathcal{C}(\mathcal{V})$ can not contain an unbounded set.
\vspace{\baselineskip}

{\bf Case 2: (There are unbounded $A,C\subseteq X_{\infty}$ with $A\overline{\phi}_{\infty}C$, $A\subseteq st(C,\mathcal{C}(\mathcal{V}))$, and $C\subseteq st(A,\mathcal{C}(\mathcal{V}))$)} By the previous case we have that $\mathcal{C}(\mathcal{V})$ does not contained unbounded sets. Therefore $X_{\infty}$ can decomposed into countably many $\mathcal{V}$-components $\{X_{\infty,i}\}_{i\in\mathbb{N}}$. Similarly, each $X_{\alpha}$ can be decomposed into countably many $\pi_{\alpha}(\mathcal{V})$-components $\{X_{\alpha,i}\}_{i\in\mathbb{N}}$. We may assume that $\pi_{\alpha}(X_{\infty,i})=X_{\alpha,i}$ for each $\alpha$ and $i$. Now, turning our attention to $A$ and $C$ we have that there is some $\alpha\in\Omega$ such that $\pi_{\alpha}(A)\overline{\phi}_{\alpha}\pi_{\alpha}(C)$. Then, without loss of generality there is an unbounded $\hat{A}\subseteq\pi_{\alpha}(A)$ such that $\hat{A}\overline{\bf b}_{\alpha}\pi_{\alpha}(C)$. By assumption, for every $x\in\hat{A}$ there is a $y_{x}\in\pi_{\alpha}(C)$ such that $x$ is connected to $y_{x}$ by a $\pi_{\alpha}(\mathcal{V})$-path. By {\bf case 1} we have that neither $\hat{A}$ nor $\pi_{\alpha}(C)$ can be $\pi_{\alpha}(\mathcal{V})$-connected. We may then assume that $\hat{A}$ can be decomposed into countably many finite $\mathcal{V}$-connected components. Say $\hat{A}=\{K_{i}\}_{i\in\mathbb{N}}$ and $\mathcal{F}_{\alpha}=\{A_{1},\ldots,A_{n}\}$. Similarly $\pi_{\alpha}(C)$ can be decomposed into $\pi_{\alpha}(C)=\{C_{i}\}_{i\in\mathbb{N}}$. Up to taking subsets we may assume that no $X_{\alpha,i}$ intersects distinct $K_{j}$'s. Then there is an injective function $\psi:\mathbb{N}\rightarrow\mathbb{N}$ such that $K_{i}$ is joined to $C_{\psi(i)}$ via a finite $\pi_{\alpha}(\mathcal{V})$-chain. As the $\pi_{\alpha}(\mathcal{V})$-chains witnessing the connections between $\pi(\mathcal{V})$ components of $\hat{A}$ and $\pi_{\alpha}(C)$ must be disjoint and bounded subsets of $X_{\alpha}$ are finite we must have that for a given $x_{0}\in X_{\alpha}$ and for every $m\in\mathbb{N}$ there is a $\pi_{\alpha}(\mathcal{V})$-chain in $X_{\alpha}$ from some $K_{i}$ to some $C_{\psi(i)}$ lying completely outside of $B(x_{0},m)$. As a uniformly bounded cover of $X_{\alpha}$ can only have finite many elements that intersect multiple $A_{i}$ we must have that there is some $N\in\mathbb{N}$ and $A_{k}$ such that if $m\geq N$ then $K_{m},C_{m}\subseteq A_{k}$. We may then assume that $\hat{A}$ is made up of those $K_{m}$ contained in $A_{k}$. Then, $f(\hat{A})\overline{\bf b}f(C)$. Then, using an argument similar to that in {\bf case 1} we can find some ${\bf b}$-cover $\{D_{1},D_{2},D_{3}\}$ with $f(\hat{A})\ll D_{1}$, $f(C)\ll D_{2}$, $f(\hat{A})\bar{\bf b}D_{2}$, $f(C)\bar{\bf b} D_{1}$, and $(f(\hat{A})\cup f(C))\bar{\bf b}D_{3}$. Then, by the fineness of our collection of coarsely canonical covers we can find an $X_{\beta}$ for which $\pi_{\beta}(\hat{A}),\pi_{\beta}(C)$ lie in different $A_{i}\in\mathcal{F}_{\beta}$. The $\pi_{\beta}(\mathcal{V})$ with have infinitely many elements that intersects multiple elements of $\mathcal{F}_{\beta}$, which is a contradiction. Therefore we must have that $A{\bf b}_{\infty}C$.
\end{proof}

Using Proposition \ref{asdim 0 characterization} we then have:

\begin{corollary}
If $\{\mathcal{F}_{\alpha}\}_{\alpha\in\Omega}$ is a directed fine collection of coarsely canonical covers of a proper discrete metric space $X$, then $asdim(X_{\infty})=0$. 
\end{corollary}

Now, we wish to show that $\Delta_{c}(X)\leq n$, then $asdim(X)\leq n$. To this end we consider our coarse proximity space $X_{\infty}$ constructed via coarse canonical covers of order not exceeding $n+1$ as above. To do this we will proceed to show that $X_{\infty}$ is is a metrizable coarse proximity space. Then we will show that $\delta$ is coarsely $n+1$-to-$1$ and employ Theorem \ref{asdim characterization}. Proving that $X_{\infty}$ is metrizable will require that instead of working with the original inverse system used to define $X_{\infty}$ we should work instead with all possible ``thickenings" of the inverse system. More specifically, let $\{\mathcal{F}_{\alpha}\}_{\alpha\in\Omega}$ be the inverse system of disjoint coarse canonical covers of $X$ that is used to define $X_{\infty}$. Say $\mathcal{F}_{\alpha}=\{A_{1}^{\alpha},\ldots,A_{\phi(\alpha)}^{\alpha}\}$. For a given $n\in\mathbb{N}$ we then define $\mathcal{F}_{\alpha,n}=\{A_{1,n}^{\alpha},\ldots,A_{\phi(\alpha),n}^{\alpha}\}$ where $A_{i,n}^{\alpha}=B(A_{i}^{\alpha},n)$. We then put a coarse proximity structure on $X_{\alpha,n}=\bigcup_{i=1}^{\phi(\alpha)}(A_{i,n}^{\alpha}\times\{i\})$ by setting $\mathcal{B}_{\alpha,n}$ to be the bornology of all finite sets and define the coarse proximity relation ${\bf b}_{\alpha,n}$ by $A{\bf b}_{\alpha,n}$ if and only if the projections of $(A\cap (A_{i,n}^{\alpha}\times\{i\}))$ and $(C\cap (A_{i,n}^{\alpha}\times\{i\}))$ into $X$ are close in $X$ with the metric coarse proximity. It is easily seen that this is a coarse proximity structure on $X_{\alpha,n}$. For $\beta\geq\alpha$ we see that $X_{\beta,n}$ surjects onto $X_{\alpha,n}$ via the inclusions $A_{i,n}^{\beta}\rightarrow A_{\rho_{\beta\alpha},n}^{\alpha}$, and that this map is a coarse proximity map. Denote these maps by $g_{\beta\alpha}$. Similarly, for each $\alpha$ we see that that there is a natural injection $h_{\alpha}:X_{\alpha}\rightarrow X_{\alpha,n}$ again by the simple inclusion $A_{i}^{\alpha}\rightarrow A_{i,n}^{\alpha}$. It is easily seen that $h_{\alpha}$ is a coarse proximity isomorphism for each $\alpha$. In summary, we have the following commutative diagram.

\begin{center}
\begin{tikzcd}
{} \arrow[r, dotted] & {X_{\gamma}} \arrow[r, "F_{\gamma\beta}"] \arrow[d, "h_{\gamma}", hook] & {X_{\beta}} \arrow[r, "f_{\beta\alpha}"] \arrow[d, "h_{\beta}", hook] & {X_{\alpha}} \arrow[d,"h_{\alpha}", hook]\\
{} \arrow[r, dotted] & {X_{\gamma,n}} \arrow[r, "g_{\gamma\beta}", two heads] & {X_{\beta,n}}\arrow[r, "g_{\beta\alpha}", two heads] & {X_{\alpha,n}}
\end{tikzcd}
\end{center}

Denoting $\varprojlim_{\alpha}X_{\alpha,n}$ by $X_{\infty,n}$ we have the following.

\begin{proposition}\label{thickening doesnt change limit}
The coarse proximity space $X_{\infty}$ coarse embeds into $X_{\infty,n}$.
\end{proposition}
\begin{proof}
Denote the projection maps from $X_{\infty}$ to each $X_{\alpha}$ by $\pi_{\alpha}$. Similarly, denote the projection maps from $X_{\infty,n}$ to $X_{\alpha,n}$ by $\rho_{\alpha}$. By construction we have that $g_{\beta\alpha}\circ h_{\beta}\circ\pi_{\beta}=h_{\alpha}\circ\pi_{\alpha}$ for each $\beta,\alpha\in\Omega$. Therefore, by the universal property of the inverse limit there is a unique coarse proximity map $h:X_{\infty}\rightarrow X_{\infty,n}$ such that $h_{\alpha}\circ\pi_{\alpha}=\rho_{\alpha}\circ h$ for each $\alpha\in\Omega$. As each $h_{\alpha}$ is an injective coarse equivalence we have that $h$ is an injective coarse proximity embedding. 
\end{proof}

\begin{corollary}
For each $n,m\in\mathbb{N}$ with $n\leq m$, the inclusion $X_{\infty,n}\rightarrow X_{\infty,m}$ is an injective coarse proximity embedding. 
\end{corollary}

\begin{proposition}\label{limit is witnessed}
The limit coarse proximity space $X_{\infty}$ is witnessed by uniform covers. 
\end{proposition}
\begin{proof}
Let $A,C\subseteq X_{\infty}$ be such that $A{\bf b}_{\infty}C$. Then $\delta(A)$ and $\delta(C)$ are close via the metric coarse proximity structure on $X$. Therefore there is an $n\in\mathbb{N}$ such that $st(\delta(A),\mathcal{V})\cap \delta(C)$ contains some unbounded set $D$, where $\mathcal{V}$ is the uniformly bounded cover of $X$ by all sets of diameter at most $n$. We then Consider $X_{\infty,3n}$. As shown above and in Proposition \ref{thickening doesnt change limit} there is an injective coarse proximity embedding $h:X_{\infty}\rightarrow X_{\infty,3n}$. For every $\alpha\in\Omega$ let $\mathcal{V}_{\alpha,3n}$ be the collection of all subsets of $X_{\alpha,3n}^{\alpha}$ of the form $(B\times\{i\})\cap (A_{i,3n}^{\alpha}\times\{i\})$ where $B$ is n element of $\mathcal{V}$. Then $\mathcal{V}_{\alpha,3n}$ is a uniformly bounded cover of $X_{\alpha,3n}$ such that $g_{\beta\alpha}(\mathcal{V}_{\beta,3n})$ refines $\mathcal{V}_{\alpha,3n}$ when $\beta\geq\alpha$. For each $\alpha$ we have that $(\rho_{\alpha}\circ h)(A){\bf b}_{\alpha,3n}(\rho_{\alpha}\circ h)(C)$ for all $\alpha$. Moreover, for each $\alpha$ this relation is witnessed by $\mathcal{V}_{\alpha,3n}$. Then, the uniformly bounded cover of $X_{\infty,3n}$ defined by $\{B\subseteq X_{\infty,3n}\mid \exists E\in\mathcal{V}_{\alpha,3n},\,\rho_{\alpha}(B)\subseteq E\}$ witnesses $h(A)$ being close to $h(C)$. Then by Proposition \ref{witnessability through isomorphisms} we have that $A$ being close to $C$ in $X_{\infty}$ is witnessed by a uniformly bounded cover of $X_{\infty}$. 

\end{proof}

Notice that in the proof of Proposition \ref{limit is witnessed} we in fact show that the coarse proximity of $X_{\infty}$ is witnessed by a countable family of uniformly bounded covers. Also, by the construction, the bornologies of $X_{\infty}$ and $X$ are isomorphic with respect to the order structure given by set inclusion. Therefore by Theorem \ref{metrizability criterion} we have the following.

\begin{corollary}\label{limit is metrizable}
The coarse proximity space $(X_{\infty},\mathcal{B}_{\infty},{\bf b}_{\infty})$ is metrizable and any metric inducing the coarse proximity structure on $X_{\infty}$ makes it into a proper metric space.
\end{corollary} 

\begin{proposition}\label{boundary map n+1 to 1}
The continuous map $\mathcal{U}\delta:X_{\infty}\rightarrow X$ is $n+1$-to-$1$.
\end{proposition}
\begin{proof}
Let $\{C_{1},\ldots,C_{k}\}$ be a collection of gradually disjoint unbounded subsets of $X_{\infty}$ such that $\{\delta(C_{1}),\ldots,\delta(C_{k})\}$ is not divergent. Assume towards a contradiction that $k>n+1$. Because the $C_{i}$ are gradually disjoint we have that $C_{i}\bar{\bf b}_{\infty}C_{j}$ for $i\neq j$. Then there is some $\alpha$ such that $\pi_{\alpha}(C_{i})\bar{\bf b}_{\alpha}\pi_{\alpha}(C_{j})$ for $i\neq j$. Because $\delta$ is a bijective function and the $\delta(C_{i})=\delta_{\alpha}(\pi_{\alpha}(C_{i}))$ are not divergent we have that each $\pi_{\alpha}(C_{i})$ must be in a different element of the cover $\mathcal{F}_{\alpha}$. However, this would imply that there are $k>n+1$ members of the cover that are not divergent, contradicting the fact that the order of every $\mathcal{F}_{\alpha}$ does not exceed $n+1$. Therefore $k\leq n+1$, which implies that $\mathcal{U}\delta$ is $n+1$-to-$1$ by Theorem \ref{finite to one criterion}. 
\end{proof}

\begin{corollary}
$asdim(X)\leq n$.
\end{corollary}
\begin{proof}
The space $X_{\infty}$ is metrizable by Corollary \ref{limit is metrizable} and $\mathcal{U}\delta$ is $n+1$-to-$1$ and surjective by Theorem \ref{finite to one criterion}. Then by Theorem \ref{metric n to one} we have that $\delta$ is coarsely $n+1$-to-$1$ and coarsely surjective, so by Theorem \ref{asdim characterization} we have that $asdim(X)\leq n$. 
\end{proof}

We then quickly have the following.

\begin{theorem}\label{coarse cofinal bounds asdim}
Given a proper metric space $X$ equipped with its metric coarse proximity structure the following inequality holds:

\[\Delta(\nu X)\leq asdim(X)\leq\Delta_{c}(X)\]
\end{theorem}

\section{Discussion and Problems}

The original problem suggested by Dranishnikov was the problem of whether or not the asymptotic dimension of an unbounded proper metric space was equal to the covering dimension of the Higson corona of the space. For general compact Hausdorff spaces $X$ the inequality $dim(X)\leq\Delta(X)$ holds. There are compact Hausdorff spaces with finite covering dimension and infinite cofinal dimension. In fact, for every natural number $n$ there is a compact Hausdorff space $X_{n}$ such that $dim(X_{n})=n$ and $\Delta(X_{n})=\infty$. Detailed constructions of such spaces can be found in \cite{pears}. 

\begin{question}
Is there a compact Hausdorff space $Z$ satisfying $dim(Z)<\Delta(Z)$ that can be realized as the Higson corona of a proper metric space?
\end{question}

We showed that for every unbounded proper metric space that $\Delta(\nu X)\leq asdim(X)\leq\Delta_{c}(X)$. We can weaken the question of whether or not $asdim(X)$ and $dim(\nu X)$ agree to the question of whether or not $\Delta(\nu X)$ and $asdim(X)$ agree. 

\begin{question}
Does the equality $\Delta(\nu X)=asdim(X)$ hold for every unbounded proper metric space?
\end{question}

Stronger questions that imply the previous are the very natural questions:

\begin{question}\label{equality}
Given an unbounded proper metric space $X$, does $\Delta(\nu X)<\infty$ imply that $\Delta_{c}(X)<\infty$?
\end{question}

\begin{question}
Does $\Delta(\nu X)=\Delta_{c}(X)$ for all unbounded proper metric spaces?
\end{question}


\begin{thebibliography}{10}

\bibitem{austinvirk}
Kyle Austin and \v{Z}iga Virk.
\newblock Higson compactification and dimension raising.
\newblock {\em Topology Appl.}, 215:45--57, 2017.

\bibitem{Bell}
G.~Bell and A.~Dranishnikov.
\newblock Asymptotic dimension.
\newblock {\em Topology and its Applications}, 155(12):1265--1296, 2008.

\bibitem{dranishnikov}
A~N Dranishnikov.
\newblock Asymptotic topology.
\newblock {\em Russian Mathematical Surveys}, 55(6):1085--1129, December 2000.

\bibitem{asdimlessthandim}
A.~N. Dranishnikov, J.~Keesling, and V.~V. Uspenskij.
\newblock On the {H}igson corona of uniformly contractible spaces.
\newblock {\em Topology}, 37(4):791--803, 1998.

\bibitem{burghelea}
Alexander Engel and Micha\l Marcinkowski.
\newblock Burghelea conjecture and asymptotic dimension of groups.
\newblock {\em J. Topol. Anal.}, 12(2):321--356, 2020.

\bibitem{Gromov}
M.~Gromov.
\newblock Asymptotic invariants of infinite groups.
\newblock In {\em Geometric group theory, {V}ol. 2 ({S}ussex, 1991)}, volume
  182 of {\em London Math. Soc. Lecture Note Ser.}, pages 1--295. Cambridge
  Univ. Press, Cambridge, 1993.

\bibitem{paper1}
Pawel Grzegrzolka and Jeremy Siegert.
\newblock Coarse proximity and proximity at infinity.
\newblock {\em Topology Appl.}, 251:18--46, 2019.

\bibitem{paper2}
Pawel Grzegrzolka and Jeremy Siegert.
\newblock Normality conditions of structures in coarse geometry and an
  alternative description of coarse proximities.
\newblock {\em Topology Proc.}, 53:285--299, 2019.

\bibitem{paper3}
Pawel Grzegrzolka and Jeremy Siegert.
\newblock Boundaries of coarse proximity spaces and boundaries of
  compactifications.
\newblock 2020.

\bibitem{hartmann}
Elisa Hartmann.
\newblock A totally bounded uniformity on coarse metric spaces.
\newblock {\em Topology Appl.}, 263:350--371, 2019.

\bibitem{isbell}
J.~R. Isbell.
\newblock {\em Uniform spaces,}.
\newblock Mathematical surveys, no. 12. American Mathematical Society,
  Providence, 1964.

\bibitem{Honari}
Sh. Kalantari and B.~Honari.
\newblock Asymptotic resemblance.
\newblock {\em Rocky Mountain J. Math.}, 46(4):1231--1262, 08 2016.

\bibitem{miyatavirk}
Takahisa Miyata and \v{Z}iga Virk.
\newblock Dimension-raising maps in a large scale.
\newblock {\em Fund. Math.}, 223(1):83--97, 2013.

\bibitem{proximityspaces}
S.~A. Naimpally and B.~D. Warrack.
\newblock {\em Proximity spaces}.
\newblock Cambridge Tracts in Mathematics and Mathematical Physics, No. 59.
  Cambridge University Press, London-New York, 1970.

\bibitem{pears}
A.~R. Pears.
\newblock {\em Dimension theory of general spaces}.
\newblock Cambridge University Press, Cambridge, England-New York-Melbourne,
  1975.

\bibitem{Roe}
John Roe.
\newblock {\em Lectures on coarse geometry}, volume~31 of {\em University
  Lecture Series}.
\newblock American Mathematical Society, Providence, RI, 2003.

\bibitem{thesis}
Jeremy Siegert.
\newblock {\em Coarse Proximity Spaces}.
\newblock PhD thesis, University of Tennessee Knoxville, 2021.

\bibitem{Smirnov}
Yu.~M. Smirnov.
\newblock On the dimension of proximity spaces.
\newblock {\em Amer. Math. Soc. Transl. (2)}, 21:1--20, 1962.

\bibitem{Yu}
Guoliang Yu.
\newblock The coarse {B}aum-{C}onnes conjecture for spaces which admit a
  uniform embedding into {H}ilbert space.
\newblock {\em Invent. Math.}, 139(1):201--240, 2000.

\end{thebibliography}

\end{document}